\documentclass[review,authoryear]{elsarticle}

\usepackage{hyperref}

\journal{Journal of \LaTeX\ Templates}

\usepackage{amssymb, amsmath,mathrsfs}
\usepackage{amsthm}
\usepackage{bbm}
\usepackage{stmaryrd}
\usepackage{hyperref}
\usepackage{enumerate}

\newcommand{\be}{\begin{eqnarray}}
\newcommand{\ee}{\end{eqnarray}}
\newcommand{\ce}{\begin{eqnarray*}}
	\newcommand{\de}{\end{eqnarray*}}
\newtheorem{theorem}{Theorem}[section]

\newtheorem{lemma}[theorem]{Lemma}
\newtheorem{remark}[theorem]{Remark}
\newtheorem{definition}[theorem]{Definition}
\newtheorem{proposition}[theorem]{Proposition}

\newtheorem{corollary}[theorem]{Corollary}
\newtheorem{assumption}[theorem]{Assumption}
\def\[{{\Big[}}
\def\]{{\Big]}}
\def\<{{\langle}}
\def\>{{\rangle}}
\def\({{\Big(}}
\def\){{\Big)}}

\def\bx{{\mathbf{x}}}

\def\bt{\begin{theorem}}
	\def\et{\end{theorem}}
\def\bl{\begin{lemma}}
	\def\el{\end{lemma}}
\def\br{\begin{remark}}
	\def\er{\end{remark}}
\def\bx{\begin{Example}}
	\def\ex{\end{Example}}
\def\bd{\begin{definition}}
	\def\ed{\end{definition}}
\def\bp{\begin{proposition}}
	\def\ep{\end{proposition}}
\def\bc{\begin{corollary}}
	\def\ec{\end{corollary}}

\def\geq{\geqslant}
\def\leq{\leqslant}

\def\bx{{\bf x}}

\usepackage{xcolor}
\allowdisplaybreaks
\usepackage{natbib}
\usepackage{filecontents}

\begin{document}

\begin{frontmatter}
\title{One-Dimensional McKean-Vlasov Stochastic Variational Inequalities and Coupled BSDEs with Locally H\"older Noise Coefficients}
%\title{McKean-Vlasov Stochastic Variational Systems with Locally H\"older Coefficients}

\author[1]{Ning Ning}
\author[2]{Jing Wu \corref{mycorrespondingauthor}}
\cortext[mycorrespondingauthor]{Corresponding author}
\ead{wujing38@mail.sysu.edu.cn}
\author[2]{Jinwei Zheng}

\address[1]{Department of Statistics, Texas A\&M University, College Station, USA.}
\address[2]{School of Mathematics, Sun Yat-sen University, Guangzhou, 510275,  China.}

%%% Group authors per affiliation:
%\author{{} {}\fnref{myfootnote}}

\begin{abstract}
In this article, we investigate three classes of equations: the McKean-Vlasov stochastic differential equation (MVSDE), the MVSDE with a subdifferential operator referred to as the McKean-Vlasov stochastic variational inequality (MVSVI), and the coupled forward-backward MVSVI. The latter class encompasses the FBSDE with reflection in a convex domain as a special case. We establish the well-posedness, in terms of the existence and uniqueness of a strong solution, for these three classes in their general forms. Importantly, we consider stochastic coefficients with locally H\"{o}lder continuity and employ different strategies to achieve that for each class.
\end{abstract}

%60G07: General theory of stochastic processes
%49J40: Variational inequalities
%60H30: Applications of stochastic analysis (to PDEs, etc.)
%	60G20  	Generalized stochastic processes

\end{frontmatter}

%\linenumbers

%\tableofcontents

\section{Introduction}

Stochastic differential equations (SDEs) of the McKean-Vlasov type are often referred to as McKean–Vlasov SDEs (MVSDEs).
They are often used in statistical physics, large-scale social interactions within the theory of mean-field games, and various other settings.  However, the dependence of the coefficients on the solution and the law of the solution introduces significant difficulties in the study of these equations. In this paper, we consider one-dimensional MVSDEs, MVSVIs, and backward MVSVIs, with increasing complexity as described in Sections \ref{intro_MVSDE},\ref{intro_MVSVI}, and \ref{intro_MVFBSDE}, respectively.
We work on a complete filtered probability space $(\Omega, \mathscr{F}, \mathbb{F}=\{\mathscr{F}_t\}_{t\geq0}, \mathbb{P})$ which supports an $\mathbb{F}$-adapted standard Brownian Motion $B$.

\subsection{McKean-Vlasov stochastic differential equations}
\label{intro_MVSDE}
We first analyze the following time-inhomogeneous MVSDE: For some $T\in(0,+\infty)$ fixed, 
\begin{equation} \label{ec0}
X_t = \xi+\int_0^t b(s, X_s, \mu_{X_s})ds+\int_0^t \sigma(s, X_s)dB_s, \quad t\in[0,T],
\end{equation}
where $\mu_{X_s}$ is the distribution of the random variable $X_s$, the drift coefficient $b: \Omega\times\mathbb{R}^+\times\mathbb{R}\times\mathcal{P}(\mathbb{R})\rightarrow\mathbb{R}$ is a measurable function with $\mathcal{P}(\mathbb{R})$ being the space of probability measures in $\mathbb{R}$, and the diffusion coefficient $\sigma: \Omega\times\mathbb{R}^+\times\mathbb{R}\rightarrow\mathbb{R}$ is a measurable function. The study of random media has had rapid development during at least the last thirty years, whose typical research approach consists of the inclusion of a random variable in the coefficients of the SDE (see, e.g. \cite{alos1999stochastic, hausenblas2007spdes, bayraktar2019controlled}). The definition of a strong solution to equation \eqref{ec0} is provided in Definition \ref{ecd01}. In this paper, we establish the well-posedness of a strong solution in Theorem \ref{th1}, under Assumption \ref{assumption1}, where we merely suppose locally Lipschitz continuity of $b$ with respect to (w.r.t.) both the state $x$ and the distribution $\mu_x$ and suppose locally H\"older continuity of $\sigma$ w.r.t. the state.

In the case of globally Lipschitz continuous diffusion coefficient and one-sided globally continuous Lipschitz drift coefficient, \cite{wang2018distribution} established the well-posedness of strong solutions. \cite{crisan2018smoothing} investigated the regularity of the solutions of MVSDEs using Malliavin calculus.
Under super-linear growth conditions, \cite{dos2019freidlin} utilized the fixed point theorem to prove the well-posedness of strong solutions.  \cite{huang2019distribution} showed the well-posedness of MVSDEs with non-degenerate diffusion under integrable conditions.
Under weaker integrability assumptions, \cite{rockner2021well} obtained strong well-posedness for MVSDEs with constant diffusion coefficient. For further MVSDEs with non-Lipschitz coefficients, see e.g.  \cite{bao2021approximations, de2020strong, hammersley2021McKean, li2023strong}.

For locally Lipschitz conditions, the existing literature is somewhat limited. The first local condition is provided in a pioneer and beautiful work \citep{carmona2015forward}, to our best knowledge. The authors provided a detailed probabilistic analysis of controlled MVSDEs, and
pointed out that the usual assumption of bounded coefficients w.r.t. the state variable precludes the application of this result to the linear quadratic models which are often used as benchmarks in stochastic control.  The first locally Lipschitz condition appears in their (B2) assumption, but in order to establish the well-posedness, they considered a simplified model whose $b$ and $\sigma$ are both in a linear form (see (B1) on page $2673$ therein). Recently, \cite{ren2023singular} proved the well-posedness for MVSDEs when the drift coefficient contains a locally integrable term and when a Lyappunov type condition is satisfied; 
\cite{liu2023tamed} established the well-posedness of equation \eqref{ec0} with $b$ being locally Lipschitzian  and $\sigma$ being H\"older continuous in the state variable. 

Different from many works, we consider the situation under much weaker conditions on the coefficients, with $b$ being locally Lipschitzian  and $\sigma$ being locally H\"older continuous in the state variable. In the Euler scheme, the truncation argument is heavily used
to handle the dependency of the locally Lipschitz constant w.r.t. to the variable. To deal with the locally H\"older diffusion coefficient, we use the Yamada-Watanabe function defined in equation \eqref{eqn:Yamada-Watanabe_func} and its properties given in Theorem \ref{thm:Properties_of_YW}. This function was introduced by \cite{yamada1971uniqueness} and we use it to address the well-posedness problem.
\vspace{-0.01cm}

\subsection{McKean-Vlasov stochastic variational inequality}
\label{intro_MVSVI}
Next, we consider the following McKean-Vlasov stochastic variational inequality (MVSVI):
	\begin{align} \label{eos1}
		X_t \in \xi+\int_0^tb(s, X_s, \mu_{X_s})ds+\int_0^t\sigma(s, X_s)dB_s-\int_0^t\partial\psi(X_s)ds,
	\end{align}
which has the same definition as equation \eqref{ec0} but with an additional subdifferential term. Here, $\psi: \mathbb{R}\rightarrow\mathbb{R}$ is a convex function and $\partial$ denotes the subdifferential operator
\begin{align}
	\label{eqn:subdifferential}
	\partial\psi(x):=\Big\{z\in \mathbb{R}: (x'-x) z\le \psi(x')-\psi(x),\;\forall x'\in\mathbb{R} \Big\}.
\end{align}
 MVSVIs generalize MVSDEs with reflection, denoted as RMVSDEs, at the boundaries of convex domains. The extra assumption on $\psi$ is provided in Assumption \ref{assumption2} which is the standard assumption.
We establish the well-posedness of a strong solution (Definition \ref{def:mvsvi}) in Theorem \ref{thos1}.

\cite{sznitman1984nonlinear} was the first to prove the well-posedness of RMVSDEs in smooth bounded domains.  Strong restrictions on the coefficients as being Lipschitz and bounded, are usually imposed.  Recently, \cite{adams2022large} proved the well-posedness of RMVSDEs in general convex domains with $b$ being superlinear growth in both space and measure. 
\cite{wang2023distribution} proved the well-posedness and established functional inequalities for RMVSDEs with singular or monotone coefficients.
 \cite{huang2022singular} established the well-posedness of singular RMVSDEs, where the drift contains a term growing linearly in space and distribution and a locally integrable term independent of distribution, while the noise coefficient is weakly differentiable in space and Lipschitz continuous in distribution w.r.t. the sum of Wasserstein and weighted variation distances. New well-posedness results and exponential ergodicity of non-dissipative RMVSDEs and singular RMVSDEs were established in  \cite{wang2023exponential} and \cite{wang2023exponentialSPA}, respectively.

Our strategy to achieve the locally H\"older condition is different from all these literatures. We apply our results of SDEs combined with the Yosida-Moreau approximation and the Yamada-Watanabe function. Though we had already applied it in other settings such as those of \cite{ning2021well, ning2023multi}, locally H\"older condition has never been achieved. After utilizing the Yamada-Watanabe function in the Yosida-Moreau approximation but incorporating some intermediate results that we obtained of equation \eqref{ec0} in the Euler scheme, the mission is accomplished.
To the best of our knowledge, this is the first instance of utilizing these two techniques in the context of SVI, whether it is for MV typed or otherwise.

%\vspace{-1cm}
\subsection{McKean-Vlasov forward-backward stochastic variational system}
\label{intro_MVFBSDE}

At last, we consider the following McKean–Vlasov forward-backward stochastic variational system (MVFBSVS), with the forward equation \eqref{eos1} and the  backward equation coupled with it:
\begin{equation} \label{eb1}
\begin{aligned}
			Y_t  &\in G(X_T, \mu_{X_T})+\int_t^T F(s, X_s, Y_s, Z_s, \mu_{X_s}, \mu_{Y_s})ds-\int_t^TZ_sdB_s\\
		&\hspace{6.5cm}-\int_t^T\partial\psi_2(Y_s)ds,
\end{aligned}
	\end{equation}	
where $\mu_{Y_s}$ is the distribution of the random variable $Y_s$.
We consider a stochastic drift coefficient  $F:\Omega\times\mathbb{R}^+\times\mathbb{R}\times\mathbb{R}\times\mathbb{R}\times\mathcal{P}(\mathbb{R})\times\mathcal{P}(\mathbb{R})\rightarrow\mathbb{R}$ as a measurable function,
 and a stochastic terminal function  $G:\Omega\times\mathbb{R}\times\mathcal{P}(\mathbb{R})\rightarrow\mathbb{R}$ as a measurable function, both of which depend on the solution of the forward equation and its distribution. Here,  $\psi_2 :\mathbb{R}\rightarrow\mathbb{R}$ is a convex function. The definition of a strong solution to equation \eqref{eb1} is provided in Definition \ref{eqn:def_MVFBSVS}. In this paper, we establish the well-posedness of a strong solution in Theorem \ref{thb2}, under additional Assumption \ref{assumption4} for the backward equation.

Forward-backward SDE (FBSDE), McKean–Vlasov FBSDE (MVFBSDE), and FBSDE with reflection have wide applications (see, e.g., \cite{ma2001reflected,carmona2018probabilistic}).
To our best knowledge, the well-posedness of MVFBSDE with reflection and the well-posedness of MVFBSVS are both unestablished. In this paper, we provide the first well-posedness result of MVFBSVS, which includes the first well-posedness result of MVFBSDE with reflection in a convex domain. In addition to the locally Lipschitzian $b$ w.r.t. both $x$ and $\mu_x$ and locally H\"older continuous $\sigma$ w.r.t. $x$, we suppose locally Lipschitzian  $F$ w.r.t. $y$, $z$ and $\mu_y$, and suppose Lipschitzian $G$ w.r.t. $x$ and $\mu_x$.

%first use truncation to show the convergence of $Y_n^m$, then we achieve the $L^1$ convergence, and use the uniform integrability to obtain the $L^2$ convergence.

\subsection{Organization of the paper}
The rest of the paper proceeds as follows. 	Section \ref{sec:Preliminaries} includes the necessary preliminaries.
In Sections \ref{sec:Well-posedness_MVSDE}, \ref{sec:Well-posedness_MVSVI}, and \ref{sec:Well-posedness_MVFBSVS},
we prove the well-posedness of strong solutions to equation \eqref{ec0}, \eqref{eos1}, and \eqref{eb1}, respectively.

\section{Preliminaries}	
\label{sec:Preliminaries}
In this section, we begin by introducing the notations that will be used throughout the paper in Section \ref{sec:Notation}. Subsequently, we provide important properties in Section \ref{sec:Properties}.

\subsection{Notations}
\label{sec:Notation}
We first introduce some common notations to be used throughout the paper, broadly classified by topics. 
\begin{itemize}
	\item Functions.
	
	\begin{itemize}
		%\item $|x|_+$ represents $|x|$ if $x\ge 0$, and $0$ if $x<0$.
		%\item
		\item Denote $\mathbbm{1}_A$ as the indicator function of set $A$, $\operatorname{Int}(A)$ as the interior of $A$, and $\overline{A}$ as the closure of $A$.
		\item For $0\leq s<t$, let $|f|_s^t$ be the  variation of $f$ on $[s,t]$.
		\item 	For $\epsilon\in(0,1), ~\delta>1$,  let $\varphi_{\epsilon,\delta}\in C(\mathbb{R}; [\frac{\epsilon}{\delta},\epsilon])$ be a symmetric function satisfying that
		$$
		0\le \varphi_{\epsilon,\delta}(x) \le \frac{2}{|x|\ln(\delta)}\quad\text{and}\quad\int_{\epsilon/ \delta}^{\epsilon}\varphi_{\epsilon,\delta}(x)dx=1.
		$$
Define $\varphi_{\epsilon,\delta}(x)=\varphi_{\epsilon,\delta}(|x|)$ for $x<0$.
		
Define the Yamada-Watanabe function
		\begin{align}
			\label{eqn:Yamada-Watanabe_func}
			V_{\epsilon,\delta}(x):=\int_0^x\int_0^y\varphi_{\epsilon,\delta}(z)dzdy;
		\end{align}
		see Theorem \ref{thm:Properties_of_YW} for its properties.
		
		\item 	For a convex function $\psi: \mathbb{R} \rightarrow \mathbb{R}$, its Yosida-Moreau approximation function \citep{barbu2010nonlinear} is defined as:
		\begin{align}
			\label{eqn:Yosida-Moreau approximation}
			\psi^n(x):=\inf\left\{\frac{n}{2}|x'-x|^2+\psi(x');\, x'\in \mathbb{R}\right\}.
		\end{align}
		Thus for every $n\geq1$, $\psi^n$ is convex and  continuously differentiable. The gradient of $\psi^n$, denoted as $\nabla\psi^n$, is monotone. Denote
		\begin{align}
			\label{eqn:Jnx}
			J_nx:=x-\frac{1}{n}\nabla\psi^n(x),
		\end{align}
		See Theorem \ref{d14} for the properties of $\psi^n$ and $J_nx$.
	\end{itemize}

	\item Measures, integrals, and spaces
	
	\begin{itemize}
		\item Let $C([0,T]; A)$ be the space of continuous functions with domain $[0,T]$ and range $A$.
		\item 	Let $\mathcal{P}(\mathbb{R})$ denote the collection of probability measures on $\mathbb{R}$. Define the Wasserstein space of order $p$, for  $p\geq1$, as
$$
\mathcal{P}_p(\mathbb{R}):=\left\{\mu\in\mathcal{P}(\mathbb{R}); \;\int_{\mathbb{R}}|x|^p\mu(dx)<\infty\right\}.
$$
  Define $W_p(\mu,\nu)$, the $p$-th order Wasserstein distance, as
		$$ W_p(\mu,\nu):=\inf\limits_{\pi\in\mathscr{C}(\mu,\nu)}\left(\int_{\mathbb{R}\times\mathbb{R}}|x-y|^pd\pi(x,y)\right)^\frac{1}{p},
		$$
		where  $\mathscr{C}(\mu,\nu)$ is the collection of coupling measures of $\mu$ and $\nu$. The space $\mathcal{P}_p(\mathbb{R})$ is a complete metric space with respect to $W_p(\mu,\nu)$.
		\item {Let  $(\Omega, \mathscr{F}, \mathbb{F}=\{\mathscr{F}_t\}_{t\geq0}, \mathbb{P})$ be a complete filtered probability space. Denote $\mathbb{S}^p[0,T]$ as the collection of continuous stochastic processes $\{X\}$ that are progressively measurable and $$\mathbb{E}\sup_{0\le s\le T}|X_s|^p< \infty.$$}
		\item Denote $\mathbb{H}^p[0,T]$ as the collection of progressively measurable processes $\{X\}$ satisfying $$\mathbb{E}\int_0^T|X_s|^pds< \infty.$$
	\end{itemize}
	
	Throughout the paper, the letter $C$, with or without subscripts, will represent a positive constant whose value may vary from line to line.
\end{itemize}

\subsection{Properties}
\label{sec:Properties}

%The following is an introduction to convex functions and Yosida-Moreau approximation functions.
%\begin{definition}[Definition and Properties of Convex Functions \cite{barbu2010nonlinear}]\label{d12}
%	A function $\psi$ defined on $\mathbb{R}^n$ is convex if and only if for all $\lambda > 0$, we have $\psi(\lambda x + (1+\lambda)y) \leq \lambda\psi(x) + (1-\lambda)\psi(y)$. Let $D:={x\in \mathbb{R}^n : \psi(x) < \infty}$, $D$ is called the effective domain of $\psi$. Convex functions on $\mathbb{R}^n$ are continuous inside $D$. If a convex function is lower semi-continuous, then the function takes finite values on any relatively closed subset of the effective domain $D$.
%\end{definition}

The following theorem covers properties of the Yosida-Moreau approximation function.
\begin{theorem}[\cite{barbu2010nonlinear}]
	\label{d14}
	The Yosida-Moreau approximation function
	$\psi^n$ defined in \eqref{eqn:Yosida-Moreau approximation} and $J_n$  defined in \eqref{eqn:Jnx} satisfy that, for any $x,y\in \mathbb{R}$,
	$$
	\left \{\begin{aligned}
		&(x-y)(\nabla\psi^n(x)-\nabla\psi^m(y))\ge -\left(\frac{1}{n}+\frac{1}{m}\right)\nabla\psi^n(x)\nabla\psi^m(y),\\
		&|\nabla\psi^n(x)-\nabla\psi^n(y)|\le n|x-y|,\\
		&\psi^n(x)=\psi(J_nx)+\frac{1}{2n}|\nabla\psi^n(x)|^2,\\
		&\psi(J_nx)\le\psi^n(x)\le\psi(x),\\
		&|J_nx-J_ny|\le|x-y|,\\
		&\lim\limits_{n\rightarrow \infty}J_nx=\Pi_{\overline{D}}(x).
	\end{aligned}\right.
	$$
	Here, $\Pi_{\overline{D}}(x)$ denotes the projection of $x$ onto $\overline{D}$, where $$D:=\{x\in \mathbb{R}:\partial\psi(x)\neq\emptyset\}.$$
\end{theorem}

The following theorem covers properties of the subdifferential operator.
\begin{theorem}[\cite{rockafellar1970maximal}]
	\label{thm:Properties_of_subdifferential}	
	The subdifferential operator $\partial\psi$ is monotone, that is, for any $x, x'\in \mathbb{R}$, $z\in \partial\psi(x), z'\in \partial\psi(x')$, we have
	$$(x-x')(z-z')\ge 0.$$
	The subdifferential operator is also maximally monotone, that is, if $x,z\in\mathbb{R}$ satisfy that
	$$
	(x-x')(z-z')\ge0,\qquad \forall x'\in \mathbb{R}, ~ z'\in \partial\psi(x'),
	$$
	then $z\in\partial\psi(x)$.
\end{theorem}	

The following theorem covers properties of the Yamada-Watanabe function.
\begin{theorem}[\cite{yamada1971uniqueness}]
	\label{thm:Properties_of_YW}
	The Yamada-Watanabe function $V_{\epsilon,\delta}(x)$ satisfies that
	\begin{equation}\label{tb12}
		\left \{\begin{aligned}
			&|x| - \epsilon \le V_{\epsilon,\delta}(x) \le |x|, \\
			&0 \le \operatorname{sgn}(x)V_{\epsilon,\delta}'(x) \le1,\\
			&0 \le V_{\epsilon,\delta}^{''}(x) \le \frac{2}{|x|\ln(\delta)}\mathbbm{1}_{[\epsilon/\delta,\epsilon]}(|x|).
		\end{aligned} \right.
	\end{equation}	
\end{theorem}

We shall also need the following lemma, which is taken from \cite[Lemma 4.6]{cepa1998problame}, to deal with the integration of functions having finite variations. 
\begin{lemma}\label{intconv}
Suppose $\{\kappa_n\}_{n\geq1}$ is a sequence of continuous functions from $[0,T]$ to $\mathbb{R}$ satisfying 
$\sup_n|\kappa_n|_0^T\leq C<\infty$ and $\kappa_n$ converges uniformly on $[0,T]$ to $\kappa$. Suppoese further $\{f_n\}$ is a sequence of continuous functions from $[0,T]$ to $\mathbb{R}$ converging uniformly to $f$. Then for all $0\leq s\leq t\leq T$, the following holds:
\begin{equation*}
\lim_{n\to\infty}\int_s^t f_n(r)d\kappa_n(r)=\int_s^t f(r)d\kappa(r).
\end{equation*}
\end{lemma}

\section{Well-posedness of the MVSDE}
\label{sec:Well-posedness_MVSDE}
%Next, we will use the explicit Euler approximation framework under the above conditions to prove the existence and uniqueness of the solution to equation \eqref{ec0} in Section \ref{sec:Existence_and_Uniqueness}. Additionally, we will provide the convergence rate of the Euler approximation framework in Section \ref{sec:Convergence_rate}.

%\subsubsection{Existence and Uniqueness of Solutions}
%\label{sec:Existence_and_Uniqueness}
In this section, we prove the existence and uniqueness of solutions  to equation \eqref{ec0} in Theorem \ref{th1}.
We impose the following conditions:
%\label{Hc1}
%\textbf{Assumption\uppercase\expandafter{3.1}}
%\label{Hc2}	
%\textbf{Assumption\uppercase\expandafter{3.2}}
%\label{Hc3}
%\textbf{Assumption\uppercase\expandafter{3.3}}
\begin{assumption}
	\label{assumption1}
	For any $x, x'\in\mathbb{R}$, $\mu, \mu'\in\mathcal{P}_1(\mathbb{R})$, $s\in (0,T]$, and $\omega\in \Omega$, suppose that
	$b(\cdot,\cdot,x,\mu)$ and $\sigma(\cdot,\cdot,x)$ are progressively measurable and there exists a constant $C>0$ such that
	$$|b(\omega, s, x, \mu)|\le C\big(1+|x|+\mu(|\cdot|)\big), \qquad |\sigma(\omega, s, x)|\le C(1+|x|),$$
	and {for any $x, ~x'\in\mathbb{R}$,
	\begin{align*}
&|b(\omega, s, x, \mu)-b(\omega, s, x', \mu')|\\
&\hspace{1.5cm}\le \Big(C\ln (e+|x|+|x'|)+\mu(|\cdot|)+\mu'(|\cdot|)\Big)\big[|x-x'|+W_1(\mu, \mu')\big], \\
	&|\sigma(\omega, s, x)-\sigma(\omega, s, x')|^2\le C\ln (e+|x|+|x'|)|x-x'|^{2\alpha+1},
\end{align*}
where $\alpha \in (0, \frac{1}{2}]$. Furthermore, the initial state $\xi$ satisfies that $\mathbb{E}|\xi|^4<\infty$.}
\end{assumption}
We first give the definition of the solution to equation \eqref{ec0}.
\begin{definition}\label{ecd01}
	A progressively measurable continuous process $X$ defined on $(\Omega, \mathscr{F}, \mathbb{F}, \mathbb{P})$ is called a strong solution to equation \eqref{ec0} if it satisfies the following conditions:
	\begin{itemize}
		\setlength\itemsep{0.15em}
		\item $\mathbb{P}(X_0=\xi)=1$.
		\item $\int_0^t\mathbb{E}|b(s, X_s, \mu_{X_s})|ds+\int_0^t\mathbb{E}|\sigma(s, X_s)|^2ds<\infty$, for any $t\in[0, T]$.
		\item $X$ satisfies that for any $t\in[0, T]$, $\mathbb{P}$-a.s.
		$$
		X_t=\xi+\int_0^tb(s, X_s, \mu_{X_s})ds+\int_0^t\sigma(s, X_s)dB_s.
		$$
	\end{itemize}
\end{definition}

The well-posedness of equation \eqref{ec0} is proved using the Euler method by first dividing $[0, T]$ into $n$ equal subintervals $[0, t_1], [t_1, t_2], \cdots, [t_{n-1}, t_n]$. Let $X^n_0=\xi$ and for $t\in [0, t_1]$, by \cite{oksendal2003stochastic} the following equation has a unique solution:
\begin{equation}\label{ec02}
	X_t^n=\xi+\int_0^tb(s, X^n_0, \mu_{\xi})ds+\int_0^t\sigma(s, X^n_0)dB_s.
\end{equation}
According to Jensen's inequality and the Burkholder-Davis-Gundy (BDG) inequality, for any $p\ge2$, we have
\begin{align*}%\label{ec03}
		&\hspace{-0.3cm}\mathbb{E}\sup\limits_{0\le s\le t_1}|X_s^n|^p\\
		& \le C\mathbb{E}|\xi|^p+C\mathbb{E}\int_0^{t_1}|b(s, X_0^n, \mu_{\xi})|^pds+\mathbb{E}\sup\limits_{0\le s\le t_1}\left|\int_0^{t_1}\sigma(s, X_0^n)dB_s\right|^p\\
		&\le C\mathbb{E}|\xi|^p+CT(1+\mathbb{E}|\xi|^p)+C\mathbb{E}\int_0^{t_1}|\sigma(s, X_0^n)|^pds\\
		&\le C\mathbb{E}|\xi|^p+CT(1+\mathbb{E}|\xi|^p)\\
		&\le C(1+\mathbb{E}|\xi|^p).
\end{align*}
Similarly, for $t\in [t_k, t_{k+1}]$ where $k=1,\ldots,n-1$, considering that $X^n$ evolves according to
\begin{equation}\label{ec2}
	dX_t^n=b(t, X_{t_k}^n, \mu_{X_{t_k}^n})dt+\sigma(t, X_{t_k}^n)dB_t,
\end{equation}
we have
$$\mathbb{E}\sup\limits_{t_k\le s\le t_{k+1}}|X_s^n|^p\le C_{n, p}(1+\mathbb{E}|\xi|^p),$$
where $C_{n, p}$ is a positive constant that only depends on $n$ and $p$.

Next, we use ${X^n}$ to find the solution of equation \eqref{ec0}.
\begin{lemma}\label{lc1}
{Assume $p \geq 2$ and $\mathbb{E}|\xi|^p<\infty$.} Under Assumption \ref{assumption1}, considering that $X_t^n$ evolves according to equations \eqref{ec02}-\eqref{ec2}, there exists a constant $C_{T,p}$ that only depends on $p$ and $T$, such that
	$$\sup_n\mathbb{E}\sup\limits_{0\le s\le T}|X_s^n|^p\le C_{p}(1+\mathbb{E}|\xi|^p).$$
\end{lemma}

\begin{proof}
Denote $k_n(s)$ as the value of $t_k$ when $s\in[t_k, t_{k+1}]$.	According to It\^o's formula,
	\begin{align*}
		(1+|X_t^n|^2)^{\frac{p}{2}} =(1+|X_0^n|^2)^{\frac{p}{2}}+\mathcal{I}_1(t)+\mathcal{I}_2(t)+\mathcal{I}_3(t)+\mathcal{I}_4(t),
	\end{align*}
	where
	\begin{align*}
		\mathcal{I}_1(t):&=p\int_0^t(1+|X_s^n|^2)^{\frac{p-2}{2}}X_s^nb(s, X_{k_n(s)}^n, \mu_{X_{k_n(s)}^n})ds,\\
		\mathcal{I}_2(t):&=\frac{p}{2}\int_0^t(1+|X_s^n|^2)^{\frac{p-2}{2}}|\sigma(s, X_{k_n(s)}^n)|^2ds,\\
		\mathcal{I}_3(t):&=p\int_0^t(1+|X_s^n|^2)^{\frac{p-2}{2}}X_s^n\sigma(s, X_{k_n(s)}^n)dB_s,\\
		\mathcal{I}_4(t):&=\frac{p(p-2)}{2}\int_0^t(1+|X_s^n|^2)^{\frac{p-4}{2}}|X_s^n|^2|\sigma(s, X_{k_n(s)}^n)|^2ds.
	\end{align*}
	Under Assumption \ref{assumption1}, using Young's inequality that
	\begin{align}
		\label{eqn:young}
		a^\frac{p-2}{2}b\le \frac{p-2}{p}a^\frac{p}{2}+\frac{2}{p}b^\frac{p}{2} \quad \text{for } a,b\geq0,
	\end{align}
	and Jensen's inequality, we have
	\begin{equation}\label{ec4}
		\begin{aligned}
			&\hspace{-0.3cm}\mathcal{I}_1(t)+\mathcal{I}_2(t)+\mathcal{I}_4(t)\\
			& \le C_p\int_0^t(1+|X_s^n|^2)^{\frac{p-2}{2}}\Big(1+|X_s^n|^2+|X_{k_n(s)}^n|^2+\mathbb{E}|X_{k_n(s)}^n|^2\Big)ds\\
			&\le C_p\int_0^t\Big[(1+|X_s^n|^2)^{\frac{p}{2}}+(1+|X_{k_n(s)}^n|^2)^{\frac{p}{2}}+(1+\mathbb{E}|X_{k_n(s)}^n|^2)^\frac{p}{2}\Big]ds.
		\end{aligned} \nonumber
	\end{equation}
	Using the BDG inequality and Assumption \ref{assumption1}, we have
	\begin{equation}\label{ec5}
		\begin{aligned}
			&\mathbb{E}\sup\limits_{0\le s\le t}|\mathcal{I}_3(s)|\le \frac{1}{2}\mathbb{E}\sup\limits_{0\le r\le t}(1+|X_r^n|^2)^\frac{p}{2}+C_p\mathbb{E}\,\mathcal{I}_2(t).
		\end{aligned} \nonumber
	\end{equation}
	Therefore, we have
	\begin{equation}\label{ec6}
		\begin{aligned}
			\mathbb{E}\sup\limits_{0\le s\le T}(1+|X_s^n|^2)^\frac{p}{2} &\le (1+\mathbb{E}|X_0^n|^2)^{\frac{p}{2}}+C_{T,p}\int_0^T\mathbb{E}\sup\limits_{0\le r\le s}(1+|X_r^n|^2)^{\frac{p}{2}}ds.
		\end{aligned} \nonumber
	\end{equation}
	Using Gr\"onwall's Lemma, we have
	\begin{equation}\label{ec7}
		\mathbb{E}\sup\limits_{0\le s\le T}|X_s^n|^p\le C_{T,p}(1+\mathbb{E}|\xi|^p).
	\end{equation}
	\end{proof}

Now we give the following existence and uniqueness result for  equation \eqref{ec0}.
\begin{theorem}\label{th1}
Assume $p \geq 2$ and $\mathbb{E}|\xi|^p<\infty$. Then under Assumption \ref{assumption1}, equation \eqref{ec0} has a unique strong solution $X$. Furthermore, $X\in \mathbb{S}^p[0, T]$, that is, there exists a constant $C_{T,p}$ depending only on $T, ~p$ such that $$\mathbb{E}\sup\limits_{0\le s\le T}|X_s|^p\le C_{T, p}(1+\mathbb{E}|\xi|^p).$$
\end{theorem}
\begin{proof}
	We complete the proof by proceeding with the following $5$ steps.
	\medskip
	
	\noindent\textbf{Step $1$.}
	Recall that $k_n(t)$ is defined as $t_k$ when $t\in[t_k, t_{k+1}]$.
	Using Lemma \ref{lc1}, the BDG inequality, and Jensen's inequality, we have that, 
	\begin{align*}
		&\mathbb{E}|X_t^n-X_{k_n(t)}^n|^p \\
		&\le 2^{p-1}\mathbb{E}\left|\int_{k_n(t)}^tb(s, X_{k_n(s)}^n, \mu_{X_{k_n(s)}^n})ds\right|^p+2^{p-1}\mathbb{E}\left|\int_{k_n(t)}^t\sigma(s, X_{k_n(s)}^n)dB_s\right|^p\\
		&\le \frac{C_p}{n^{p-1}}\int_{k_n(t)}^t\mathbb{E}\left(1+|X_{k_n(s)}^n|^{p}+\mathbb{E}|X_{k_n(s)}^n|^{p}\right)ds\\
		&\quad+\frac{C_p}{n^{\frac{p}{2}-1}}\left(\int_{k_n(t)}^t\mathbb{E}(1+|X_{k_n(s)}^n|^{p}+\mathbb{E}|X_{k_n(s)}^n|^{p})ds\right)\\
		&\le \frac{C_p}{n^\frac{p}{2}},
	\end{align*}
	which yields that
	\begin{align}
		\label{ec18}
		\sup\limits_{0\le s\le T}\mathbb{E}|X_{k_n(s)}^n-X_s^n|^p\le \frac{C_p}{n^\frac{p}{2}}.
	\end{align}
	\smallskip
	
	{
		\noindent\textbf{Step $2$.}
	We now estimate $\mathbb{E}|Z_t^{m,n}|$ where $$Z_t^{m,n}:=X_t^{m}-X_t^n$$ is assumed. Here, we are going to use the Yamada-Watanabe function $V_{\epsilon, \delta}(x)$ defined in equation \eqref{eqn:Yamada-Watanabe_func}; see Theorem \ref{thm:Properties_of_YW} for its properties. According to It\^o's formula, we have
	\begin{equation}\label{ec11}
		\begin{aligned}
			V_{\epsilon, \delta}(Z_t^{m,n}) &= \int_0^tV_{\epsilon, \delta}'(Z_s^{m,n})\Big(b(s, X_{k_n(s)}^n, \mu_{X_{k_n(s)}^n})-b(s, X_{k_{m}(s)}^{m}, \mu_{X_{k_{m}(s)}^{m}})\Big)ds\\
			&\quad+\frac{1}{2}\int_0^tV_{\epsilon, \delta}^{''}(Z_s^{m,n})\big|\sigma(s, X_{k_n(s)}^n)-\sigma(s, X_{k_m(s)}^m)\big|^2ds\\
			&\quad+\int_0^tV_{\epsilon, \delta}'(Z_s^{m,n})\big(\sigma(s, X_{k_n(s)}^n)-\sigma(s, X_{k_{m}(s)}^{m})\big)dB_s\\
			&=:\, \mathcal{I}_{1, 1}(t)+\mathcal{I}_{1, 2}(t)+\mathcal{I}_{1, 3}(t).
		\end{aligned} \nonumber
	\end{equation}
	By Lemma \ref{lc1}, $\mathbb{E}\,\mathcal{I}_{1,3}(t)=0$.
	Denote, for any $R>0$,
	$$\Omega^{m,n}_R:=\left\{\omega: \sup\limits_{0\le s\le T}|X_s^n|\vee \sup\limits_{0\le s\le T}|X_s^m|> R\right\}.$$
	Under Assumption \ref{assumption1} and by the properties of the Wasserstein distance  and equation (\ref{ec7}), denoting $$c_0:=\sup_n\mathbb{E}\sup_{t\leq T}|X^n_t|\quad\text{and}\quad L_R:=C\ln (e+2R),$$ we obtain that
	\begin{align*}
		\mathcal{I}_{1, 1}(t)
		&\le \int_0^t|V_{\epsilon, \delta}'(Z_s^{m,n})|\Big[(L_R+2c_0)\Big(|Z^{m,n}_s|+|X_{k_n(s)}^n-X_s^n|+|X_{k_{m}(s)}^m-X_s^{m}|\\
&\hspace{3cm}+\mathbb{E}|Z_s^{m,n}|+\mathbb{E}|X_{k_n(s)}^n-X_s^n|+\mathbb{E}|X_{k_{m}(s)}^{m}-X_s^{m}|\Big)\Big]ds\\
		&\quad+ C\int_0^t|V_{\epsilon, \delta}'(Z_s^{m,n})|\Big(1+2c_0+|X_{k_n(s)}^n|+|X_{k_{m}(s)}^{m}|\Big) \mathbbm{1}_{\Omega^{m,n}_R}ds\\
		&\le (L_R+2c_0)\int_0^t\Big[|Z_s^{m,n}|+|X_{k_n(s)}^n-X_s^n|+|X_{k_{m}(s)}^{m}-X_s^{m}|\\
&\hspace{2.5cm}
	+\mathbb{E}|X_{k_n(s)}^n-X_s^n|+\mathbb{E}|X_{k_{m}(s)}^{m}-X_s^{m}|+\mathbb{E}|X_s^n-X_s^{m}|\Big]ds\\
		&\quad+C\int_0^t\Big(1+2c_0+|X_{k_n(s)}^n|+|X_{k_{m}(s)}^{m}|\Big) \mathbbm{1}_{\Omega^{m,n}_R}ds.
	\end{align*}
	Under Assumption \ref{assumption1}, setting $\delta=2$, we have
	%\begin{equation}\label{ec13}
	\begin{align*}
		\mathcal{I}_{1, 2}(t)&\le \int_0^tV_{\epsilon, \delta}^{''}(Z_s^{m,n})\Big[\big|\sigma(s, X_{k_n(s)}^n)-\sigma(s, X_{k_{m}(s)}^{m})\big|^2\mathbbm{1}_{\Omega\backslash\Omega^{m,n}_R}\\
		&\quad\hspace{2.6cm}+\big|\sigma(s, X_{k_n(s)}^n)-\sigma(s, X_{k_{m}(s)}^{m})\big|^2\mathbbm{1}_{\Omega^{m,n}_R}\Big]ds\\
		&\le C\int_0^t\epsilon^{2\alpha}L_Rds+\frac{C}{\epsilon}\int_0^t\Big(1+|X_{k_n(s)}^n|^2+|X_{k_{m}(s)}^{m}|^2\Big)\mathbbm{1}_{\Omega^{m,n}_R}ds
		\\
		&\quad+\frac{C}{\epsilon}\int_0^tL_R\Big[|X_{k_n(s)}^n-X_s^n|^{2\alpha+1}
		+|X_{k_{m}(s)}^{m}-X_s^{m}|^{2\alpha+1}\Big]ds.
	\end{align*}
	%\nonumber
	%\end{equation}
	Using Lemma \ref{lc1}, we have
	\begin{equation}\label{ec14}
		\begin{aligned}
			\mathbb{E}\,\mathcal{I}_{1, 2}(t)&\le \frac{CL_R}{\epsilon}\int_0^t\Big(\mathbb{E}|X_{k_n(s)}^n-X_s^n|^{2\alpha+1}
+\mathbb{E}|X_{k_{m}(s)}^{m}-X_s^{m}|^{2\alpha+1}\Big)ds\\
			&\quad+\frac{C}{\epsilon}\int_0^t\mathbb{E}\Big[(1+|X_{k_n(s)}^n|^2+|X_{k_{m}(s)}^{m}|^2)\mathbbm{1}_{\Omega^{m,n}_R}\Big]ds+CtL_R\epsilon^{2\alpha}	.
		\end{aligned} \nonumber
	\end{equation}
	By H\"older's inequality, Chebyshev's inequality and Lemma \ref{lc1}, we have
	\begin{equation}\label{ec15}
		\begin{aligned}
			&\hspace{-0.8cm}\mathbb{E}\left[\int_0^T(1+|X_{k_n(s)}^n|^2+|X_{k_{m}(s)}^{m}|^2)\mathbbm{1}_{\Omega^{m,n}_R}ds\right]\le \frac{c_{T}}{R^{2}}(1+\mathbb{E}|\xi|^{4}),
		\end{aligned}
	\end{equation}
	and furthermore, 
		\begin{align}\label{ec16}
			\mathbb{E}|Z_t^{m,n}|
			&\le (L_R+c_0)\int_0^t\Bigg[\mathbb{E}|Z_s^{m,n}|+\mathbb{E}|X_{k_n(s)}^n-X_s^n|+\mathbb{E}|X_{k_{m}(s)}^{m}-X_s^{m}|\Bigg]ds\nonumber\\
			&\quad+\frac{CL_R}{\epsilon}\int_0^t\Bigg[\mathbb{E}|X_{k_n(s)}^n-X_s^n|^{2\alpha+1}
+\mathbb{E}|X_{k_{m}(s)}^{m}-X_s^{m}|^{2\alpha+1}\Bigg]ds\nonumber\\
			&\quad+\frac{C_T(1+\mathbb{E}|\xi|^{4})}{R^2}(1+\frac{1}{\epsilon})+CtL_R\epsilon^{2\alpha}+\epsilon.
		\end{align}
	Let $l=\frac{2}{1+2\alpha}, ~R>1$ and $\epsilon=\frac{1}{R^l}$.  Using Gr\"onwall's Lemma and equation \eqref{ec18}, we have
		\begin{align}\label{ec019}
			\mathbb{E}|Z_t^{m,n}|
			&\le C\Bigg[(L_R+c_0)\left({n}^{-1/2}+{m}^{-1/2}\right)+L_RR^{l}\left({n}^{-\alpha-1/2}+{m}^{-\alpha-1/2}\right)\nonumber\\
			&\hspace{4cm}+\frac{C(1+R^l)}{R^2}+{L_RR^{-2\alpha l}}\Bigg](1+R)^{ct}.
		\end{align}
Then for $t_0>0$ satisfying that $ct_0<\frac{4\alpha}{1+2\alpha}$, 
\begin{equation}\label{ec0191} \begin{aligned}
			\sup_{t\leq t_0}\mathbb{E}|Z_t^{m,n}|\to 0, \quad\mbox{by ~letting} ~ m, n\to\infty, ~~\mbox{and ~then} ~~R\to\infty.
\end{aligned}
	\end{equation}}	\smallskip
	
	{\noindent\textbf{Step $3$.}
	We now prove that $\{X^n; n\geq1\}$ is a Cauchy sequence in $\mathbb{S}^2[0, t_0]$. By It\^o's formula,
	\begin{equation}\label{ec020}
		\begin{aligned}	
			|Z_t^{m,n}|^2=\mathcal{I}_{2,1}(t)+\mathcal{I}_{2,2}(t)+\mathcal{I}_{2,3}(t),
		\end{aligned}\nonumber
	\end{equation}
	where 
	\begin{align*}
		\mathcal{I}_{2,1}(t):&=2\int_0^tZ_s^{m,n}\Big(b(s, X_{k_n(s)}^n, \mu_{X_{k_n(s)}^n})-b(s, X_{k_{m}(s)}^{m}, \mu_{X_{k_{m}(s)}^{m}})\Big)ds,\\
		\mathcal{I}_{2,2}(t):&=\int_0^t|\sigma(s, X_{k_n(s)}^n)-\sigma(s, X_{k_{m}(s)}^{m})|^2ds,\\
		\mathcal{I}_{2,3}(t):&= 2\int_0^tZ_s^{m,n}\big(\sigma(s, X_{k_n(s)}^n)-\sigma(s, X_{k_{m}(s)}^{m})\big)dB_s.
	\end{align*}
	Under Assumption \ref{assumption1}, by Young's inequality \eqref{eqn:young},  H\"older inequality, and Lemma \ref{lc1}, together with equations \eqref{ec18} and \eqref{ec15}, we have
	%	\begin{equation}\label{ec021}
		\begin{align*}	
			\mathbb{E}|\mathcal{I}_{2,1}(t)|
			&\le 2\mathbb{E}\int_0^t(L_R+c_0)|Z_s^{m, n}|\Big[|X_{k_n(s)}^n-X_{k_{m}(s)}^{m}|
			+\mathbb{E}|X_{k_n(s)}^n-X_{k_{m}(s)}^{m}|\Big]ds\\
			&\quad+C\mathbb{E}\int_0^t|Z_s^{m,n}|\Big(1+2c_0+|X_{k_n(s)}^n|+|X_{k_{m}(s)}^{m}|\Big) \mathbbm{1}_{\Omega^{m,n}_R}ds\\
			&\le 3(L_R+c_0)\int_0^t\mathbb{E}|Z_s^{m,n}|^2ds+\frac{C_{T, c_0}}{R^{2}}(1+\mathbb{E}|\xi|^4)\\
			&\quad+C(L_R+c_0)\int_0^t\Bigg[\mathbb{E}|X_{k_n(s)}^n-X_s^n|^2+\mathbb{E}|X_{k_{m}(s)}^{m}-X_s^m|^2\Bigg]ds\\
&\quad+(L_R+c_0)\int_0^t\Big(\mathbb{E}|Z_s^{m,n}|+\mathbb{E}|X_{k_n(s)}^n-X_s^n|+\mathbb{E}|X_{k_{m}(s)}^{m}-X_s^{m}|\Big)^2ds.
		\end{align*}
		%	\end{equation}
	It follows from the BDG inequality that
	\begin{equation}\label{ec19}
		\begin{aligned}	
			\mathbb{E}\sup_{0\le s\le t}|\mathcal{I}_{2,3}(s)|&\le \frac{1}{2}\mathbb{E}\sup_{0\le s\le t}|Z_s^{m,n}|^2+C\mathbb{E}\,\mathcal{I}_{2,2}(t).
		\end{aligned}\nonumber
	\end{equation}
	Under Assumption \ref{assumption1}, by Young's inequality that
	\begin{align}	
		\label{eqn:young2}
		|a|^{1+2\alpha}\le {2\alpha}|a|^2+(1-2\alpha)|a|,
	\end{align}
	 and by equation \eqref{ec15}, we have
	\begin{align*}	
		\mathbb{E}|\mathcal{I}_{2,2}(t)|&\le \mathbb{E}\int_0^tL_R|X_{k_n(s)}^n-X_{k_{m}(s)}^{m}|^{1+2\alpha}ds\\
		&\quad+C\mathbb{E}\int_0^t\Big(2+|X_{k_n(s)}^n|^2+|X_{k_{m}(s)}^{m}|^2\Big)\mathbbm{1}_{\Omega^{m,n}_R}ds\\
		&\le 3\int_0^tL_R\mathbb{E}|Z_s^{m,n}|^2ds+\int_0^tL_R\mathbb{E}|Z_s^{m,n}|ds+\frac{c_{T}}{R^2}(1+\mathbb{E}|\xi|^4)\\
		&\quad+3\int_0^tL_R\Big[\mathbb{E}|X_{k_n(s)}^n-X_s^n|^2+\mathbb{E}|X_{k_{m}(s)}^{m}-X_s^{m}|^2\Big]ds\\
		&\quad+ L_R\int_0^t\Big[\mathbb{E}|X_{k_n(s)}^n-X_s^n|
+\mathbb{E}|X_{k_{m}(s)}^{m}-X_s^{m}|\Big]ds.
	\end{align*}
	Summing up, by equations \eqref{ec18} and \eqref{ec019}-\eqref{ec0191}, we have
	\begin{align*}	
		\mathbb{E}\sup_{0\le s\le t_0}|Z_s^{m,n}|^2
		&\le  6(L_R+c_0)\int_0^{t_0}\mathbb{E}\sup_{0\le r\le s}|Z_r^{m,n}|^2ds+\frac{C_{T}}{R^{2}}(1+\mathbb{E}|\xi|^4)\\
		&\quad+ CL_R{t_0}\Big[\sup_{t\leq t_0}\mathbb{E}|Z_t^{m,n}|+n^{-1/2}+m^{-1/2}
\Big].
	\end{align*}
	By Gr\"onwall's Lemma,
	\begin{align*}	%\label{ec22}
		&\mathbb{E}\sup_{0\le s\le t_0}|Z_s^{m,n}|^2\\
		&\le C(1+R)^{ct_0}\Big[L_R{t_0}\big(\sup_{t\leq t_0}\mathbb{E}|Z_t^{m,n}|+n^{-1/2}+m^{-1/2}
\big)+\frac{C_{T}}{R^{2}}(1+\mathbb{E}|\xi|^4)\Big].
			\end{align*}
	Then by equation \eqref{ec0191} and noting that $ct_0<2$, we get by sending $R\to\infty$,
	\begin{equation}\label{ec240}
		\begin{aligned}	
			\lim_{m, n\to\infty}\mathbb{E}\sup_{t\in[0,t_0]}|Z^{m,n}_t|^2=0.
		\end{aligned}
	\end{equation}}
	\smallskip

	{\noindent\textbf{Step $4$.}  By equation \eqref{ec240}, $\{X^n\}$ is a Cauchy sequence on $\mathbb{S}^2[0, t_0]$, and then there exists $X\in \mathbb{S}^2[0, t_0]$ satisfying that
	\begin{equation}\label{ec25}
		\lim\limits_{n\rightarrow \infty}\mathbb{E}\sup_{0\le s\le t_0}|X_s^n-X_s|^2=0.
	\end{equation}
	By Lemma \ref{lc1}, we know that $X\in\mathbb{S}^p([0,t_0])$ for all $p\geq2$ satisfying $$\mathbb{E}\sup\limits_{0\le t\le t_0}|X_t|^p\le C(1+\mathbb{E}|\xi|^p).$$
	Now our goal is to show that $X$ is a solution to equation \eqref{ec0} on $[0, t_0]$.
	By the properties of the Wasserstein distance, we have
	\begin{equation}\label{ec26}
		\lim\limits_{n\rightarrow \infty}\sup_{0\le s\le t_0}W_1(\mu_{X_s^n}, \mu_{X_s})\le \lim\limits_{n\rightarrow \infty}\mathbb{E}\sup_{0\le s\le t_0}|X_s^n-X_s|=0.\nonumber
	\end{equation}
	Notice that for any $t\in [0, t_0]$,
	\begin{equation}\label{ec27}
		\begin{aligned}
			\lim\limits_{n\rightarrow \infty}W_1(\mu_{X_{k_n(t)}^n}, \mu_{X_t})
			&\le \lim\limits_{n\rightarrow \infty}W_1(\mu_{X_{k_n(t)}^n}, \mu_{X_t^n})+\lim\limits_{n\rightarrow \infty}W_1(\mu_{X_t^n}, \mu_{X_t})\\
			&\le \lim\limits_{n\rightarrow \infty}\mathbb{E}|X_{k_n(t)}^n-X_t^n|+\lim\limits_{n\rightarrow \infty}W_1(\mu_{X_t^n}, \mu_{X_t})\\
			&\le \lim\limits_{n\rightarrow \infty}\frac{C}{\sqrt{n}}+\lim\limits_{n\rightarrow \infty}W_1(\mu_{X_t^n}, \mu_{X_t})\\
			&=0.
		\end{aligned}\nonumber	
	\end{equation}
	Applying the condition on $b$ and Lemma \ref{lc1} yields
	\begin{align*}	
		&\mathbb{E}\int_0^{t_0}\Big|b(s, X_{k_n(s)}^n, \mu_{X_{k_n(s)}^n})-b(s, X_s, \mu_{X_s})\Big|ds\\
		&\le 2\mathbb{E}\int_0^{t_0}\Big(L_R+\mathbb{E}|X_{k_n(s)}^n|+\mathbb{E}|X_{s}|\Big)\Big(|X_{k_n(s)}^n-X_s|+W_1(\mu_{X_{k_n(s)}^n}, \mu_{X_s})\Big)ds\\
&\quad+2\mathbb{E}\int_0^{t_0}\Big(1+|X_{k_n(s)}^n|+|X_s|+\mathbb{E}|X_{k_n(s)}^n|+\mathbb{E}|X_s|\Big)\\
&\hspace{6cm}\times\mathbbm{1}_{\big\{\sup_{t\leq T}(|X^n_t|\vee|X_t|)>R\big\}}ds\\
		&\le C(1+L_R)\int_0^{t_0}\mathbb{E}|X_{k_n(s)}^n-X_s|ds +\frac{C}{R}(1+\mathbb{E}|\xi|^{2}),
	\end{align*}
	and then we have by  equations \eqref{ec18}  and  \eqref{ec25} that
	\begin{equation}\label{ec29}
		\lim\limits_{n\rightarrow \infty}\mathbb{E}\int_0^{t_0}\big|b(s, X_{k_n(s)}^n, \mu_{X_{k_n(s)}^n})-b(s, X_s, \mu_{X_s})\big|ds=0.
	\end{equation}
	Similarly,
	\begin{equation*}%\label{ec30}
		\lim\limits_{n\rightarrow \infty}\mathbb{E}\int_0^{t_0}\big|\sigma(s, X_{k_n(s)}^n)-\sigma(s, X_s)\big|^{2}ds=0,
	\end{equation*}	
and moreover,
	\begin{equation}\label{ec31}
		\lim\limits_{n\rightarrow \infty}\mathbb{E}\sup_{0\le t\le {t_0}}\left|\int_0^t\Big(\sigma(s, X_{k_n(s)}^n)-\sigma(s, X_s)\Big)dW_s\right|=0.\nonumber
	\end{equation}	
	Hence, $X$ is a solution to equation \eqref{ec0} on $[0,t_0]$.} 
	\bigskip
	
{\noindent\textbf{Step $5$.}
In this step we prove the uniqueness of the solution on $[0, t_0]$.
	Suppose that both $X$ and $\widetilde{X}$ are solutions to equation \eqref{ec0}, then by the second property in Definition \ref{ecd01} and the BDG inequality, we have \begin{equation}\label{1mom}\mathbb{E}\sup\limits_{0\le s\le t_0}|X_s|<\infty.\end{equation}	
	Denote
	\begin{equation}\label{ec32}
		\tau_N:=\inf\limits_{0\le t\le t_0}\big\{t:|X_t|> N\big\}\land t_0.\nonumber
	\end{equation}		
	By It\^o's formula,
	\begin{equation}\label{ec33}
		\begin{aligned}
			(1+|X_{t\land \tau_N}|^2)^{\frac{p}{2}} &= (1+|\xi|^2)^{\frac{p}{2}}
			+p\int_0^{t\land \tau_N}(1+|X_s|^2)^{\frac{p-2}{2}}X_sb(s, X_s, \mu_{X_s})ds\\
			&\quad+\frac{p}{2}\int_0^{t\land \tau_N}(1+|X_s|^2)^{\frac{p-2}{2}}|\sigma(s, X_s)|^2ds\\
			&\quad+p\int_0^{t\land \tau_N}(1+|X_s|^2)^{\frac{p-2}{2}}X_s\sigma(s, X_s)dB_s\\
			&\quad+\frac{p(p-2)}{2}\int_0^{t\land \tau_N}(1+|X_s|^2)^{\frac{p-4}{2}}|X_s|^2|\sigma(s, X_s^n)|^2ds\\
			&=:(1+|\xi|^2)^{\frac{p}{2}}+\mathcal{I}_{3,1}(t)+\mathcal{I}_{3,2}(t)+\mathcal{I}_{3,3}(t)+\mathcal{I}_{3,4}(t).
		\end{aligned} \nonumber
	\end{equation}	
	It follows from the BDG inequality that
	\begin{equation}\label{ec34}
		\begin{aligned}
			&\mathbb{E}\sup\limits_{0\le s\le  t_0\land \tau_N}|\mathcal{I}_{3,3}(s)|\le \frac{1}{2}\mathbb{E}\sup\limits_{0\le r\le  t_0\land \tau_N}(1+|X_r|^2)^\frac{p}{2}+C_p\mathbb{E}\,\mathcal{I}_{3,2}(t_0\land \tau_N).
		\end{aligned} \nonumber
	\end{equation}
	Under Assumption \ref{assumption1}, by equation \eqref{1mom}, we have
	\begin{equation}\label{ec35}
		\begin{aligned}
			&\hspace{-0.3cm}\mathbb{E}|\mathcal{I}_{3,1}(t)+\mathcal{I}_{3,2}(t)+\mathcal{I}_{3,4}(t)|\\ &\le C_p\mathbb{E}\int_0^t(1+|X_{s\land \tau_N}|^2)^{\frac{p-2}{2}}\Big(1+|X_{s\land \tau_N}|^2+(\mathbb{E}|X_s|)^2\Big) \mathbbm{1}_{\{s\le \tau_N\}}ds\\
			&\le C_p\int_0^t\mathbb{E}(1+|X_{s\land \tau_N}|^2)^{\frac{p}{2}}ds.
		\end{aligned} \nonumber
	\end{equation}
		By Gr\"onwall's Lemma, we have
	\begin{equation}\label{ec37}
		\begin{aligned}
			\mathbb{E}\sup\limits_{0\le s\le t_0\land \tau_N}(1+|X_s|^2)^{\frac{p}{2}}\le C_p(1+\mathbb{E}|\xi|^p),
		\end{aligned} \nonumber
	\end{equation}
which yields by sending $N\to\infty$ that 
		\begin{equation}\label{momentX}
			\mathbb{E}\sup\limits_{0\le s\le t_0}(1+|X_s|^2)^{\frac{p}{2}}\le C_p(1+\mathbb{E}|\xi|^p).
		\end{equation}
	Similarly,
		\begin{equation}\label{momentX1}
			\mathbb{E}\sup\limits_{0\le s\le t_0}(1+|\widetilde{X}_s|^2)^{\frac{p}{2}}\le C_p(1+\mathbb{E}|\xi|^p).
		\end{equation}
	Applying It\^o's formula to $V_{\epsilon, \delta}(X_t-\widetilde{X}_t)$, similar to equation \eqref{ec16},  by equations \eqref{momentX} and \eqref{momentX1} and Chebyshev's inequality,
	we have for $0\leq t\leq t_0$,
		\begin{align*}
			&\mathbb{E}|X_t-\widetilde{X}_t|\\
			&\le \epsilon+\mathbb{E}V_{\epsilon, \delta}(X_t-\widetilde{X}_t)\\
&\le  \epsilon+\mathbb{E}\int_0^t\big(L_R+2c_1\big)\Big[|X_s-\widetilde{X}_s|+\mathbb{E}|X_s-\widetilde{X}_s|\Big]ds
+C\epsilon^{2\alpha}L_Rt\\
&\quad+C\Big(1+\frac1\epsilon\Big)\mathbb{E}\int_0^t
\Big[1+|X_s|^2+|\widetilde{X}_s|^2+2c_1\Big]ds
\mathbbm{1}_{\{\sup_{t\leq  t_0}(|X_t|\vee|\widetilde{X}_t|)>R\}}\\&\le  \epsilon+\mathbb{E}\int_0^t\big(L_R+2c_1\big)\Big[|X_s-\widetilde{X}_s|+\mathbb{E}|X_s-\widetilde{X}_s|\Big]ds\\
&\quad+C\epsilon^{2\alpha}L_Rt+C_{T}R^{-2}\Big(1+\frac1\epsilon\Big)(1+\mathbb{E}|\xi|^4).
		\end{align*}
where $c_1:=\mathbb{E}\sup_{t\leq t_0}(|X_t|\vee|\widetilde{X}_t|)$. 	Using Gr\"onwall's Lemma, we have for $t\in[0, t_0]$,
	\begin{equation}\label{ec44}
		\begin{aligned}
			\mathbb{E}|X_t-\widetilde{X}_t|
			\leq C_{T}(1+R)^{2t}\Bigg[\epsilon^{2\alpha}L_Rt+R^{-2}\Big(1+\frac1\epsilon\Big)(1+\mathbb{E}|\xi|^4)\Bigg].
\end{aligned} \nonumber
\end{equation}
Set  $\epsilon^{2\alpha}=(1+R)^{-2t_0-1}$.  Sending $R\rightarrow\infty$, we have for every $t\in[0,t_0]$,
	\begin{equation}\label{ec45}	\begin{aligned}
		\mathbb{E}|X_t-\widetilde{X}_t|=0.\end{aligned}
	\end{equation}	
	It then follows that equation \eqref{ec0} has a unique strong solution on $[0,t_0]$.}
	\bigskip

	{\noindent\textbf{Step $6$.} By Steps $1-5$, we have proved the well-posedness of equation \eqref{ec0} on $[0,t_0]$. For the case of $t_0<T$, suppose there exists some constant $k>1$ satisfying $k t_0\geq T$. Then we may repeat the arguments in Steps $1-5$ and prove inductively the well-posedness on $[(i-1)t_0, i t_0\wedge T]$ where $i=1, \cdots, k$. Hence, we have proved that equation \eqref{ec0} has a unique strong solution $X$ on $[0,T]$ and it satisfies that
\begin{equation*}
\mathbb{E}\sup_{t\in[0,T]}|X_t|^p\leq C_{p,T}(1+\mathbb{E}|\xi|^p).
\end{equation*}}   
\end{proof}

\section{Well-posedness of the MVSVI}
\label{sec:Well-posedness_MVSVI}
In this section, we prove the existence and uniqueness of solutions  to equation \eqref{eos1} in Theorem \ref{thos1}.
We impose the following conditions.
\begin{assumption}	\label{assumption2}
	{$\psi:\mathbb{R}\to[0, +\infty)$ is a proper, lower semi-continuous, convex function with $\psi(0)=0$ and 
	$0\in \operatorname{Int}(D)$, and $\mathbb{P}$-a.s. $\xi\in \overline{D}$ where $D=\{x: \partial\psi(x)\neq \emptyset\}$. }
\end{assumption}
We first give the definition of the solution to equation \eqref{eos1}.
\begin{definition}
	\label{def:mvsvi}
	A pair of progressively measurable continuous processes $(X, \phi)$ defined on $(\Omega, \mathscr{F}, \{\mathscr{F}_t\}_{t\geq0}, \mathbb{P})$ is called a solution to equation \eqref{eos1}, if $(X, \phi)$ satisfies the following conditions:
	\begin{itemize}
		\item $\mathbb{P}(X_0=\xi)=1$.
				\item $\int_0^t\mathbb{E}|b(s, X_s, \mu_{X_s})|ds+\int_0^t\mathbb{E}|\sigma(s, X_s)|^2ds<\infty$, for any $0<t\le T$.
		\item For any $0<s<t\le T$,
		$$
		X_t=\xi+\int_0^tb(s, X_s, \mu_{X_s})ds+\int_0^t\sigma(s, X_s)dB_s-\phi_t,\quad \mathbb{P}-a.s..$$
		\item For every $t\in[0,T]$, $\mathbb{P}$-a.s.  $X_t\in\overline{D}$, and $\phi$ is of bounded variation satisfying that $\phi(0)=0$, and for any $\varrho\in C([0,T];\overline{D})$ and $0<s<t\le T$,
		\begin{equation}\label{eos995}	
			\int_s^t(\varrho_u-X_u)d\phi_u+\int_s^t\psi(X_u)du\le \int_s^t\psi(\varrho_u)du, ~~~\mathbb{P}-a.s..
		\end{equation}
	\end{itemize}
\end{definition}

To faciliate mathematical derivation, we list the following useful properties.
\begin{remark}\label{ros1}
	%By equation \eqref{eos995}, letting $\varrho(u)=0$, we have that for any $0\le %s<t\le T$,
	%$$
	%\int_s^tX_ud\phi_u\le\int_s^t\psi(X_u)du, ~~~a.s..
	%$$
	Suppose both $(X, \phi)$ and $(X', \phi')$ are solutions of equation \eqref{eos1}, by equation \eqref{eos995}, we have
	$$
	\int_s^t(X_u-X'_u)(d\phi_u-d\phi'_u)\ge 0, ~~~a.s..
	$$
	%If $\{a: |a|\le a_0\}\subset \operatorname{Int}(D)$, then for any fixed $\omega$ %and any $0\le s<t\le T$, we have
	%\begin{equation}\label{eos096}
		%a_0|\phi|_s^t\le \int_s^tX_ud\phi_u+M(t-s), \qquad M=\sup_{|x|\le a_0}|\psi(x)|.
	%\end{equation}
%    For any $s=b_0<b_1<\cdots<b_{n-1}<b_n=t$, letting $\varrho_u=a$ where $a$ is any value satisfying $|a|\le a_0$, it follows from \eqref{eos995} that
%	$$
%	a_0|\phi_{b_{k}}-\phi_{b_{k-1}}|\le \int_{b_{k-1}}^{b_k}X_ud\phi_u+M(b_{k-1}-b_k).
%	$$
%	Furthermore,
%	$$
%	a_0\sum_{k=1}^{n}|\phi_{b_{k}}-\phi_{b_{k-1}}|\le \int_s^tX_ud\phi_u+M(t-s).
%	$$
%	By the arbitrariness of the partition points and the definition of $|\phi|_s^t$, we have
%	$$
%	a_0|\phi|_s^t\le \int_s^tX_ud\phi_u+M(t-s).
%	$$
\end{remark}
The following theorem is the main result of this section.
\begin{theorem}\label{thos1}
	{Suppose $\mathbb{E}\big[|\xi|^p+|\psi(\xi)|^2\big]<\infty$ for any $p\geq2$.} Then under Assumptions \ref{assumption1} and \ref{assumption2}, equation \eqref{eos1} has a unique solution $(X, \phi)$. Moreover, there exists a constant $C_{p,T}$ that depends only on $p$ and $T$, such that $$\mathbb{E}\sup\limits_{0\le s\le T}|X_s|^p\le C_{p,T}(1+\mathbb{E}|\xi|^p).$$
\end{theorem}
\begin{proof}
	The proof is proceed in the following $6$ steps, using the properties of the Yosida-Moreau approximation function (Theorem \ref{d14}).
	\bigskip
	
	\noindent\textbf{Step $1$.} We first consider an approximating stochastic process $X_t^n$ evolving according to the following SDE:
	\begin{equation}\label{eos3}
		\begin{aligned}
			X_t^n = \xi+\int_0^tb(s, X_s^n, \mu_{X_s^n})ds+\int_0^t\sigma(s, X_s^n)dB_s-\int_0^t\nabla\psi^n(X_s^n)ds,
		\end{aligned}
	\end{equation}
	where $\psi^n(x)$ is the Yosida-Moreau approximation of $\psi(x)$.
	That is, in equation \eqref{eos3}, we replace the subdifferential $\partial\psi$ in equation \eqref{eos1} with the gradient $\nabla\psi^n$.
	For every $n\geq1$, according to Theorems \ref{d14} and \ref{th1}, equation \eqref{eos3} has a unique strong solution denoted by $X^n$, which satisfies $$\mathbb{E}\sup_{0\leq s\leq T}|X_s^n|^p\leq C_{n,p,T}(1+\mathbb{E}|\xi|^p).$$
	Now we aim to show that $C_{n,p,T}$ in the above equation is in fact independent of $n$.
	Under Assumptions \ref{assumption1} and \ref{assumption2}, by It\^o's formula,
	\begin{align}\label{eos4}
		|X_t^n|^2 &= |\xi|^2+2\int_0^tX_s^nb(s, X_s^n, \mu_{X_s^n})ds+2\int_0^tX_s^n\sigma(s, X_s^n)dB_s\nonumber\\
		&\quad-2\int_0^tX_s^n\nabla\psi^n(X_s^n)ds+\int_0^t|\sigma(s, X_s^n)|^2ds\nonumber\\
		&\le |\xi|^2+C\int_0^t\Big(1+|X_s^n|^2+(\mathbb{E}|X_s^n|)^2\Big)ds+2\int_0^tX_s^n\sigma(s, X_s^n)dB_s\nonumber\\
		&\quad-2\int_0^tX_s^n\nabla\psi^n(X_s^n)ds\nonumber\\
		&=:|\xi|^2+\mathcal{J}_{1,1}(t)+\mathcal{J}_{1,2}(t)+\mathcal{J}_{1,3}(t).
	\end{align}
	By the BDG inequality and Young's inequality, under Assumptions \ref{assumption1} and \ref{assumption2},
	\begin{equation}\label{eos5}
		\begin{aligned}
			\mathbb{E}\sup_{0\le s\le t}\mathcal{J}_{1,2}(s)
			&\le C\mathbb{E}\left(\int_0^t|X_s^n|^2|\sigma(s, X_s^n)|^2ds\right)^\frac{1}{2}\\
			&\le \frac{1}{2}\mathbb{E}\sup_{0\le s\le t}|X_s^n|^2+C\int_0^t\mathbb{E}(1+|X_s^n|^2)ds.
		\end{aligned}\nonumber
	\end{equation}
	Since $\nabla\psi^n (x)$ is monotonically increasing, we know from Assumption \ref{assumption1} and $\psi(0)=0$ that $$\mathcal{J}_{1,3}(t)=-2\int_0^tX_s^n\nabla\psi^n(X_s^n)ds\leq 0,$$ and then
	\begin{equation}\label{eos7}
		\begin{aligned}
			\mathbb{E}\sup_{0\le s\le t}|X_s^n|^2\le C(1+\mathbb{E}|\xi|^2)+C\int_0^t\sup_{0\le r\le s}\mathbb{E}|X_r^n|^2ds.
		\end{aligned}\nonumber
	\end{equation}
	Then by Gr\"onwall's lemma
	\begin{equation}\label{eos8}
		\begin{aligned}
			\sup_n\mathbb{E}\sup_{0\le s\le t}|X_s^n|^2\le C(1+\mathbb{E}|\xi|^2).
		\end{aligned}\nonumber
	\end{equation}
	Again by It\^o's formula, for $p>2$,
	\begin{align}\label{eos9}
		(1+|X_t^n|^2)^\frac{p}{2}
		&=(1+|\xi|^2)^\frac{p}{2}+p\int_0^t(1+|X_s^n|^2)^\frac{p-2}{2}X_s^nb(s, X_s^n, \mu_{X_s^n})ds\nonumber\\
		&\quad+\frac{p}{2}\int_0^t(1+|X_s^n|^2)^\frac{p-2}{2}|\sigma(s, X_s^n)|^2ds\nonumber\\
		&\quad+p\int_0^tX_s^n(1+|X_s^n|^2)^\frac{p-2}{2}\sigma(s, X_s^n)dB_s\nonumber\\
		&\quad-p\int_0^t(1+|X_s^n|^2)^\frac{p-2}{2}X_s^n\nabla\psi^n(X_s^n)ds\nonumber\\
		&\quad +\frac{p(p-2)}{2}\int_0^t(1+|X_s|^2)^{\frac{p-4}{2}}|X_s|^2|\sigma(s, X_{s})|^2ds\nonumber\\
		&=:(1+|\xi|^2)^\frac{p}{2}+\mathcal{J}_{2,1}(t)+\mathcal{J}_{2,2}(t)+\mathcal{J}_{2,3}(t)+\mathcal{J}_{2,4}(t)+\mathcal{J}_{2,5}(t).
	\end{align}
	Under Assumptions \ref{assumption1} and \ref{assumption2}, by the BDG inequality and Young's inequality \eqref{eqn:young},
	we have
	\begin{align}\label{eos10}
		&\mathbb{E}\sup_{0\le s\le t}\Big|\mathcal{J}_{2,1}(s)+\mathcal{J}_{2,2}(s)+\mathcal{J}_{2,3}(s)+\mathcal{J}_{2,5}(s)\Big|\nonumber\\
		&\le C_p\mathbb{E}\int_0^t(1+|X_s^n|^2)^\frac{p-2}{2}\Big(1+|X_s^n|^2+(\mathbb{E}|X_s^n|)^2\Big)ds+\frac{1}{2}\mathbb{E}\sup_{0\le s\le t}(1+|X_s^n|^2)^\frac{p}{2}\nonumber\\
		&\le  C_p\mathbb{E}\int_0^t(1+|X_s^n|^2)^\frac{p}{2}ds+C_p\int_0^t(\mathbb{E}|X_s^n|)^pds+
		+\frac{1}{2}\mathbb{E}\sup_{0\le s\le t}(1+|X_s^n|^2)^\frac{p}{2}.
	\end{align}
	Since $x\nabla\psi^n(x)\geq 0$, we have $$\mathcal{J}_{2,4}(t)=-p\int_0^t(1+|X_s^n|^2)^\frac{p-2}{2}X_s^n\nabla\psi^n(X_s^n)dt\le 0.$$
	Then plugging the above equation and equation \eqref{eos10} into equation \eqref{eos9}, by Young's inequality \eqref{eqn:young}, we have
	\begin{equation}\label{eos17}
		\begin{aligned}	
			\mathbb{E}\sup_{0\le s\le t}(1+|X_s^n|^2)^\frac{p}{2}
			&\le  C_p\mathbb{E}(1+|\xi|^2)^\frac{p}{2}+C_p\int_0^t\mathbb{E}\sup_{0\le r\le s}(1+|X_r^n|^2)^\frac{p}{2}ds.
		\end{aligned}\nonumber
	\end{equation}
	Gr\"onwall's Lemma yields that there exsits a constant $C_{p,T}>0$ independent of $n$ satisfying that
	\begin{equation}\label{eos18}
		\sup_{n}\mathbb{E}\sup_{0\le s\le T}|X_s^n|^{p}\le C_{p,T}\mathbb{E}(1+|\xi|^p).
	\end{equation}	
\smallskip
	
	\noindent\textbf{Step $2$.}  In this step, we will prove that $$\mathbb{E}\sup\limits_{0\leq s\leq T}|\nabla\psi^n(X_s^n)|^4\leq Cn^\frac{7}{4}.$$
	Since $\psi^n(x)$ is convex and  $0\in\mathrm{Int} (D)$, there exists some $a_0>0$ such  that $\{a: |a|\le a_0\}\subset \operatorname{Int}(D)$. Then, for any $x\in \mathbb{R}$, we have
	\begin{equation*}%\label{eos998}
		(a-x)\nabla\psi^n(x)\le \psi^n(a)-\psi^n(x)\le \psi^n(a)\le \psi(a).
	\end{equation*}
	Setting $M=\sup\limits_{|a|\le a_0}\psi(a)$, then by the inequality above, we have
	\begin{equation}\label{eos997}
		a_0|\nabla\psi^n(x)|\le x\nabla\psi^n(x)+M, \hspace{2em} \forall x\in \mathbb{R}.
	\end{equation}
	Using equations \eqref{eos4} and \eqref{eos997}, we obtain
	\begin{align*}
		&\hspace{-0.3cm}\left(a_0\int_0^t|\nabla\psi^n(X_s^n)|ds\right)^p\\
		&\le 2^{p-1}\left|\int_0^tX_s^n\nabla\psi^n(X_s^n)ds\right|^p+2^{p-1}M^Pt^p\\
		&\le 2^{p-1}|\xi|^{2p}+2^{p-1}\left|\int_0^tX_s^nb(s, X_s^n, \mu_{X_s^n})ds\right|^p+2^{p-1}\left|\int_0^tX_s^n\sigma(s, X_s^n)dB_s\right|^p\\
		&\quad+2^{p-1}\left|\int_0^t|\sigma(s, X_s^n)|^2ds\right|^p+2^{p-1}M^Pt^p+2^{p-1}|X_t^n|^{2p}.
	\end{align*}
	Therefore, under Assumptions \ref{assumption1} and \ref{assumption2}, by the BDG inequality and equation \eqref{eos18}, we obtain
	\begin{equation}\label{eos19}
		\sup_{n}\mathbb{E}\left(\int_0^T|\nabla\psi^n(X_s^n)|ds\right)^p\le C_{p,T,M,a_0}(1+\mathbb{E}|\xi|^p).
	\end{equation}
	
	Note that $\psi^n\geq 0$. Using It\^o's formula and the properties of the Yosida-Moreau approximation function, we have
	\begin{align}\label{eqn:step2}
		&|\psi^n(X_t^n)|^2\nonumber\\
		&\le|\psi^n(\xi)|^2+2\int_0^t\psi^n(X_s^n)\nabla\psi^n(X_s^n)b(s, X_s^n, \mu_{X_s^n})ds\nonumber\\
		&\quad+2\int_0^t\psi^n(X_s^n)\nabla\psi^n(X_s^n)\sigma(s, X_s^n)dB_s+\int_0^t|\nabla\psi^n(X_s^n)|^2|\sigma(s, X_s^n)|^2ds\nonumber\\
		&\quad+n\int_0^t\psi^n(X_s^n)|\sigma(s, X_s^n)|^2ds-2\int_0^t\psi^n(X_s^n)|\nabla\psi^n(X_s^n)|^2ds\nonumber\\
		&\le |\psi^n(\xi)|^2+2n\int_0^t\psi^n(X_s^n)|X_s^n||b(s, X_s^n, \mu_{X_s^n})|ds\nonumber\\
		&\quad+2\int_0^t\psi^n(X_s^n)\nabla\psi^n(X_s^n)\sigma(s, X_s^n)dB_s+3n\int_0^t\psi^n(X_s^n)|\sigma(s, X_s^n)|^2ds\nonumber\\
		&\quad-2\int_0^t\psi^n(X_s^n)|\nabla\psi^n(X_s^n)|^2ds\nonumber\\
		&=:|\psi^n(\xi)|^2+\mathcal{J}_{3,1}(t)+\mathcal{J}_{3,2}(t)+\mathcal{J}_{3,3}(t)+\mathcal{J}_{3,4}(t).
	\end{align}
	By the BDG inequality,
	\begin{align}	\label{eos21}
		\mathbb{E}\sup_{0\le s\le t}\mathcal{J}_{3,2}(s)
		&\le 2\mathbb{E}\left(\int_0^t\Big|\psi^n(X_s^n)\nabla\psi^n(X_s^n)\sigma(s, X_s^n)\Big|^2ds\right)^\frac{1}{2}\nonumber\\
		&\le \frac{1}{2}\mathbb{E}\sup_{0\le s\le t}|\psi^n(X_s^n)|^2+4n\mathbb{E}\int_0^t|\psi^n(X_s^n)||\sigma(s, X_s^n)|^2ds\nonumber\\
		&= \frac{1}{2}\mathbb{E}\sup_{0\le s\le t}|\psi^n(X_s^n)|^2+\frac43\mathbb{E}\,\mathcal{J}_{3,3}(t).		
	\end{align}
	By the inequality that $|\psi^n(X_s^n)|\le |\nabla\psi^n(X_s^n)||X_s^n|$ and Young's inequality that $ab\le \eta a^3+C_{\eta}b^\frac{3}{2},$ under Assumptions \ref{assumption1} and \ref{assumption2}, we have
	\begin{align*}%	\label{eos22}
		&\mathcal{J}_{3,1}(t)+\frac73\mathcal{J}_{3,3}(t) \nonumber\\
		&\le Cn\int_0^t|\psi^n(X_s^n)|^\frac{1}{3}|\nabla\psi^n(X_s^n)|^\frac{2}{3}|X_s^n|^\frac{2}{3}\Big(|X_s^n||b(s, X_s^n, \mu_{X_s^n})|+|\sigma(s, X_s^n)|^2\Big)ds\nonumber\\
		&\le  \int_0^t\psi^n(X_s^n)|\nabla\psi^n(X_s^n)|^2ds+C_{1}n^\frac{3}{2}\int_0^t|X_s^n|\Big(1+|X_s^n|^2+(\mathbb{E}|X_s^n|)^2\Big)^\frac{3}{2}ds.
	\end{align*}
	Plugging the above equation and equation \eqref{eos21} into equation \eqref{eqn:step2}, reorganizing the terms, it follows from  equation \eqref{eos18} that
	\begin{equation}\label{eos23}
		\begin{aligned}	
			\mathbb{E}\sup_{0\le s\le t}|\psi^n(X_s^n)|^2+\mathbb{E}\int_0^t\psi^n(X_s^n)|\nabla\psi^n(X_s^n)|^2ds
			&\le C\mathbb{E}|\psi^n(\xi)|^2+Cn^\frac{3}{2}\\
			&\le C(\mathbb{E}|\psi(\xi)|^2+n^\frac{3}{2}).
		\end{aligned}\nonumber
	\end{equation}
	Therefore,
	\begin{equation}\label{eos24}
		\mathbb{E}\sup_{0\le s\le T}|\psi^n(X_s^n)|^2\le C(\mathbb{E}|\psi(\xi)|^2+n^\frac{3}{2}).\nonumber
	\end{equation}
	Using the properties of the Yosida-Moreau approximation function, we can further conclude that
	\begin{equation}\label{eos25}
		\mathbb{E}\sup_{0\le s\le T}|\nabla\psi^n(X_s^n)|^4\le4n^2\mathbb{E}\sup_{0\le s\le T}|\psi^n(X_s^n)|^2\le C(n^\frac{7}{2}+n^2\mathbb{E}|\psi(\xi)|^2).
	\end{equation}
	\smallskip
	
	\noindent\textbf{Step $3$.} In this step, we will use the Yamada-Watanabe function $V_{\epsilon, \delta}(x)$ and Theorem \ref{thm:Properties_of_YW} to obtain $\mathbb{E}|X_t^{n,m}|$,
	where
	$$X_t^{n,m}:=X_t^n-X_t^m.$$
	 We have the following equation generated by It\^o's formula:
	\begin{align*}	
		|X_t^{n,m}|&\le V_{\epsilon, \delta}(X_t^{n,m})+\epsilon\\
		&\le \int_0^t V_{\epsilon, \delta}'(X_s^{n,m})\Big(b(s, X_s^n, \mu_{X_s^n})-b(s, X_s^m, \mu_{X_s^m})\Big)ds\\
		&\quad+\frac{1}{2}\int_0^tV_{\epsilon, \delta}^{''}(X_s^{n,m})\Big|\sigma(s, X_s^n)-\sigma(s, X_s^m)\Big|^2ds\\
		&\quad+ \int_0^t V_{\epsilon, \delta}'(X_s^{n,m})\Big(\sigma(s, X_s^n)-\sigma(s, X_s^m)\Big)dB_s\\
		&\quad-\int_0^t V_{\epsilon, \delta}'(X_s^{n,m})\Big(\nabla\psi^n(X_s^n)-\nabla\psi^m(X_s^m)\Big)ds+\epsilon\\
		&=: \mathcal{J}_{4,1}(t)+\mathcal{J}_{4,2}(t)+\mathcal{J}_{4,3}(t)+\mathcal{J}_{4,4}(t)+\epsilon.
	\end{align*}
	Set $$\Omega^{n, m}:=\left\{\omega\in \Omega: \sup\limits_{0\le s\le T}|X_s^n(\omega)|\vee\sup\limits_{0\le s\le T}|X_s^m(\omega)|>R\right\}.$$
	Denote for $p>0$,
	\begin{equation}\label{eos27}
	{c_{p}}:=\sup_n\mathbb{E}\sup_{t\leq T}|X^n_t|^p.
	\end{equation}
By equation \eqref{eos18} and Chebyshev's inequality,
	under Assumptions \ref{assumption1} and \ref{assumption2}, 
	\begin{equation}\label{eos28}
		\begin{aligned}	
			\mathbb{E}|\mathcal{J}_{4,1}(t)|&\le \mathbb{E}\int_0^t\big(L_R+c_1\big)\big[|X_s^{n,m}|+\mathbb{E}|X_s^{n,m}|\big]ds\\
			&\quad+C\mathbb{E}\int_0^t\Big(1+|X_s^n|+|X_s^m|+\mathbb{E}|X_s^n|+\mathbb{E}|X_s^m|\Big)\mathbbm{1}_{\Omega^{n, m}}ds\\
			&\le 2\int_0^t\big(L_R+c_1\big)\mathbb{E}|X_s^{n,m}|ds+\frac{C_{T}(1+c_2)^{1/2}c_q^{1/2}}{R^{q/2}},
		\end{aligned}\nonumber
	\end{equation}	
where $q>2$. 
	Under Assumptions \ref{assumption1} and \ref{assumption2}, using the properties of the Yamada-Watanabe function, and by equation \eqref{eos18}, we can conclude that
	\begin{equation}\label{eos29}
		\begin{aligned}	
			\mathbb{E}|\mathcal{J}_{4,2}(t)|&\le \frac{\epsilon^{2\alpha}}{\ln\delta}L_R^2t
			+\frac{C}{\epsilon}\mathbb{E}\int_0^t\Big(1+|X_s^n|^2+|X_s^m|^2\Big)\mathbbm{1}_{\Omega^{n, m}}ds\\
			&\le\frac{\epsilon^{2\alpha}}{\ln\delta} L_R^2 t+\frac{C_T}{\epsilon}\frac{(1+c_4)^{1/2}c_q^{1/2}}{R^{q/2}}.
		\end{aligned}\nonumber
	\end{equation}
	By equation \eqref{eos18}, we know that $\mathcal{J}_{4,3}(t)$ is a martingale, and hence
	$$\mathbb{E}\,\mathcal{J}_{4,3}(t)= \mathbb{E}\int_0^t V_{\epsilon, \delta}'(X_s^{n,m})\Big(\sigma(s, X_s^n)-\sigma(s, X_s^m)\Big)dB_s=0.$$
	Using the properties of the Yamada-Watanabe function in Theorem \ref{thm:Properties_of_YW}, the properties of the Yosida-Moreau approximation function, H\"older's inequality, and Jensen's inequality, we obtain
	\begin{align*}	
		\mathbb{E}\,\mathcal{J}_{4,4}(t)
		&=-\mathbb{E}\int_0^t V_{\epsilon, \delta}'(X_s^{n,m})\Big(\nabla\psi^n(X_s^n)-\nabla\psi^m(X_s^m)\Big)ds\\
		&=-\mathbb{E}\int_0^TV_{\epsilon, \delta}'(|X_s^{n,m}|)\frac{X_s^{n,m}}{|X_s^{n,m}|}\Big(\nabla\psi^n(X_s^n)-\nabla\psi^m(X_s^m)\Big)ds\\
		&\le C\mathbb{E}\int_0^t\frac{\delta}{\epsilon}\left(\frac{1}{n}+\frac{1}{m}\right)|\nabla\psi^n(X_s^n)||\nabla\psi^m(X_s^m)|ds\\
		&\le \frac{C\delta}{\epsilon}\Bigg[\frac{1}{n}\Big(\mathbb{E}\sup_{0\le s\le T}|\nabla\psi^n(X_s^n)|^2\Big)^\frac{1}{2}\Big(\mathbb{E}\Big(\int_0^T|\nabla\psi^m(X_s^m)|ds\Big)^2\Big)^\frac{1}{2}\\
		&\qquad\quad+\frac{1}{m}\Big(\mathbb{E}\sup_{0\le s\le T}|\nabla\psi^m(X_s^m)|^2\Big)^\frac{1}{2}\Big(\mathbb{E}\Big(\int_0^T|\nabla\psi^n(X_s^n)|ds\Big)^2\Big)^\frac{1}{2}\Bigg]\\
		&\le \frac{C\delta}{\epsilon}\Bigg[\frac{1}{n}\Big(\mathbb{E}\sup_{0\le s\le T}|\nabla\psi^n(X_s^n)|^4\Big)^\frac{1}{4}\mathbb{E}\Big(\int_0^T|\nabla\psi^m(X_s^m)|ds\Big)^2\Big)^\frac{1}{2}\\
		&\qquad\quad+\frac{1}{m}\Big(\mathbb{E}\sup_{0\le s\le T}|\nabla\psi^m(X_s^m)|^4\Big)^\frac{1}{4}\Big(\mathbb{E}\Big(\int_0^T|\nabla\psi^n(X_s^n)|ds\Big)^2\Big)^\frac{1}{2}\Bigg]\\
		&\le \frac{C\delta}{\epsilon}\left(n^{-\frac{1}{8}}+m^{-\frac{1}{8}}\right).
	\end{align*}
	Taking $\delta=2$,	\begin{equation}\label{eos31}
		\begin{aligned}	
			\mathbb{E}|X_t^{n,m}|&\le \epsilon+2\int_0^t\big(L_R+c_1\big)\mathbb{E}|X_s^{n,m}|ds+C_{T,q}\frac{R^{-{q/2}}}{\epsilon}\\
			&\quad+C_1t{L_R}{\epsilon}^{2\alpha }+C_1\frac1{\epsilon}\left(n^{-\frac{1}{8}}+m^{-\frac{1}{8}}\right).
		\end{aligned}\nonumber
	\end{equation}	
	By Gr\"onwall's Lemma,  we obtain
	\begin{equation}\label{eos42}
		\begin{aligned}	
			\mathbb{E}|X_t^{n,m}|\le C_{T,q}(1+R)^{2t}\Bigg[\frac{R^{-{q/2}}}{\epsilon}+{L_R}{\epsilon}^{2\alpha }+\frac1{\epsilon}\left(n^{-\frac{1}{8}}+m^{-\frac{1}{8}}\right)\Bigg].
		\end{aligned}
	\end{equation}
	\smallskip
	
	\noindent\textbf{Step $4$.} 	In this step, we will prove that $\{X^n\}$ is a Cauchy sequence in $\mathbb{S}^p[0,T]$. Let $0<\theta<1$. By It\^o's formula
	\begin{align*}	
		|X_t^{n,m}|^{p}
		&= p\int_0^t|X_s^{n,m}|^{p-2}X_s^{n,m}\Big(b(s, X_s^n, \mu_{X_s^n})-b(s, X_s^m, \mu_{X_s^m})\Big)ds\\
		&\quad+\frac{p(p-1)}{2}\int_0^t|X_s^{n,m}|^{p-2}\Big|\sigma(s, X_s^n)-\sigma(s, X_s^m)\Big|^2ds\\
		&\quad+p\int_0^t|X_s^{n,m}|^{p-2}X_s^{n,m}\Big(\sigma(s, X_s^n)-\sigma(s, X_s^m)\Big)dB_s\\
		&\quad-p\int_0^|X_s^{n,m}|^{p-2}X_s^{n,m}\Big(\nabla\psi^n(X_s^n)-\nabla\psi^m(X_s^m)\Big)ds\\
				&=\mathcal{J}_{5,1}(t)+\mathcal{J}_{5,2}(t)+\mathcal{J}_{5,3}(t)+\mathcal{J}_{5,4}(t).
	\end{align*}
	Under Assumptions \ref{assumption1} and \ref{assumption2}, by Young's inequality \eqref{eqn:young}, we have
	\begin{align*}	
		\mathcal{J}_{5,1}(t)
		&\le p\int_0^t|X_s^{n,m}|^{p-2}\big(L_R+c_1\big)\Big[|X_s^{n,m}|^2+|X_s^{n,m}|\mathbb{E}|X_s^{n,m}|\Big]ds\\
		&\quad+C p\int_0^t|X_s^{n,m}|^{p-1}\Big(2+|X_s^n|+|X_s^m|+\mathbb{E}|X_s^n|+\mathbb{E}|X_s^m|\Big)\mathbbm{1}_{\Omega^{n, m}}ds\\
		&\le (2p-1)\big(L_R+c_1\big)\int_0^t|X_s^{n,m}|^{p}ds+\big(L_R+c_1\big)\int_0^t(\mathbb{E}|X_s^{n,m}|)^pds\\
		&\quad+C_p\int_0^t\Big(2+|X_s^n|+|X_s^m|+\mathbb{E}|X_s^n|+\mathbb{E}|X_s^m|\Big)^p ds\mathbbm{1}_{\Omega^{n, m}}.
	\end{align*}
	Then  by equation \eqref{eos18} and H\"older's inequality,
	\begin{equation}\label{eos35}
		\begin{aligned}	
			\mathbb{E}\sup_{t\leq T}|\mathcal{J}_{5,1}(t)|
			&\le (2p-1)\big(L_R+c_1\big)\int_0^T\mathbb{E}|X_s^{n,m}|^{p}ds\\
&\quad+\big(L_R+c_1\big)\int_0^t(\mathbb{E}|X_s^{n,m}|)^pds+\frac{C_{p,q,T}}{R^{q/2}}.
		\end{aligned}\nonumber
	\end{equation}
	By the BDG inequality,
	\begin{equation}\label{eos36}
		\begin{aligned}	
			\mathbb{E}\sup_{0\le s\le t}|\mathcal{J}_{5,3}(s)|&\le \frac{1}{2}\mathbb{E}\sup_{0\le s\le t}|X_s^{n,m}|^{p}+C_p\mathbb{E}\,\mathcal{J}_{5,2}(t).
		\end{aligned}\nonumber
	\end{equation}	
	Under Assumptions \ref{assumption1} and \ref{assumption2}, by Young inequalities \eqref{eqn:young} and \eqref{eqn:young2}, we have
	%\begin{equation}\label{eos37}
	\begin{align*}	
		\mathbb{E}\sup_{t\leq T}\mathcal{J}_{5,2}(t)
		&=\frac{p(p-1)}{2}\mathbb{E}\int_0^T|X_s^{n,m}|^{p-2}\Big|\sigma(s, X_s^n)-\sigma(s, X_s^m)\Big|^2ds\\
				&\le \frac{p(p-1)}{2}L_R\mathbb{E}\int_0^T|X_s^{n,m}|^{p-1+2\alpha}ds\\
		&\quad+C_p\mathbb{E}\int_0^T|X_s^{n,m}|^{p-2}\Big(1+|X_s^n|^2+|X_s^m|^2\Big)\mathbbm{1}_{\Omega^{n, m}}ds\\
		&\le \frac{p(p-1)}{2}L_R\int_0^T\mathbb{E}|X_s^{n,m}|^{p}ds+C_pL_R\int_0^T\mathbb{E}|X_s^{n,m}|ds+\frac{C_{p,q,T}}{R^{q/2}}.
	\end{align*}
	%\end{equation}	
	By properties of the Yosida-Moreau approximation function, {for any $t\in[0,T]$,
\begin{align*}	
\mathcal{J}_{5,4}(t)&=-p\int_0^t |X_s^{n,m}|^{p-2}X_s^{n,m}\Big(\nabla\psi^n(X_s^n)-\nabla\psi^m(X_s^m)\Big)ds\\
&\le p\int_0^t |X_s^{n,m}|^{p-2}\left(\frac{1}{n}+\frac{1}{m}\right)|\nabla\psi^n(X_s^n)||\nabla\psi^m(X_s^m)|ds.
\end{align*}
Then by H\"older's inequality and Jensen's inequality, we obtain
	\begin{align*}	
		&\mathbb{E}\sup_{t\leq T}\mathcal{J}_{5,4}(t)\\
		&\le C_p\mathbb{E}\int_0^T|X_s^{n,m}|^{p-2}\left(\frac{1}{n}+\frac{1}{m}\right)\nabla\psi^n(X_s^n)\nabla\psi^m(X_s^m)ds\\
	&\le \frac{C_p}{n}\left(\mathbb{E}\sup_{0\le s\le T}|\nabla\psi^n(X_s^n)|^4\right)^\frac{1}{4}\left(\mathbb{E}\left(\int_0^T|X_s^{n,m}|^{p-2}|\nabla\psi^m(X_s^m)|ds\right)^2\right)^\frac{1}{2}\\
		&\quad+\frac{C_p}{m}\left(\mathbb{E}\sup_{0\le s\le T}|\nabla\psi^m(X_s^m)|^4\right)^\frac{1}{4}\left(\mathbb{E}\left(\int_0^T|X_s^{n,m}|^{p-2}|\nabla\psi^n(X_s^n)|ds\right)^2\right)^\frac{1}{2}\\
&\le C_{p,T}\left(n^{-\frac{1}{8}}+m^{-\frac{1}{8}}\right),
	\end{align*}}
	where, the last inequality holds, since we have by applying equations \eqref{eos18} and \eqref{eos19} and H\"older's inequality that
	\begin{equation}\label{eos39}
		\begin{aligned}	
			&\mathbb{E}\left(\int_0^T|X_s^{n,m}|^{p-2}|\nabla\psi^m(X_s^m)|ds\right)^2	\\
			&\le C_p\sqrt{\mathbb{E}\sup_{0\le s\le T}\Big(1+|X_s^n|^{4p-8}+|X_s^m|^{4p-8}\Big)}\sqrt{\mathbb{E}\left(\int_0^t|\nabla\psi^m(X_s^m)|ds\right)^4}\\
			&\le C_{p,T}.
		\end{aligned}\nonumber
	\end{equation}	
	Hence,  by equation \eqref{eos42},
	\begin{align*}	
		\mathbb{E}\sup_{0\le s\le t}|X_s^{n,m}|^{p}&\le c'{p^2}\int_0^t\big(L_R+c_1\big)\mathbb{E}|X_s^{n,m}|^{p}ds\\
		&\quad+C_pL_R\int_0^t\Big[\mathbb{E}|X_s^{n,m}|+(\mathbb{E}|X_s^{n,m}|)^p\Big]ds\\
		&\quad+\frac{C_{T,p,q}}{R^{q/2}}+C_{p,T}\left(n^{-\frac{1}{8}}+m^{-\frac{1}{8}}\right).
	\end{align*}
Take $\epsilon=(1+R)^{-l_0}$ with $l_0:=\frac{(c'p^2+2)T+2}{2\alpha}$. By Gr\"onwall's Lemma and sending $\theta \rightarrow 0$, we obtain
\begin{align*}	
\mathbb{E}\sup_{0\le s\le t}|X_s^{n,m}|^{p}&\le C_{p,q,T}\Bigg[L_R(1+R)^{l_0+2t+c'p^2t}\left(R^{-\frac{q}{2}}+n^{-\frac{1}{8}}+m^{-\frac{1}{8}}\right)+\frac{L_R^2}{R^2}\Bigg]\\
		&\quad+C_{q,T}\left(R^{-\frac{q}{2}}+n^{-\frac{1}{8}}+m^{-\frac{1}{8}}\right).
	\end{align*}
Choosing $q$ sufficiently large (for example, take $q$ satisfying  $R^{\frac{q}{2}-1}>(1+R)^{l_0+2T+c'p^2T}$), so that by letting $n, m\to\infty$ first and then $R\to\infty$ we have
	\begin{align}	\label{eos45}
		\lim_{n,m\to\infty}\mathbb{E}\sup_{0\le s\le T}|X_s^{n,m}|^p=0.
	\end{align}
Hence, $\{X^n\}$ is a Cauchy sequence on $\mathbb{S}^p[0, T]$.
	\bigskip

	\noindent\textbf{Step $5$.} 	
	Next, we prove the existence of a solution $(X, \phi)$ to equation \eqref{eos1}. Since ${X^n}$ is a Cauchy sequence on $\mathbb{S}^p[0, T]$, there exists ${X}\in\mathbb{S}^p[0, T]$ such that
	\begin{equation}\label{pconX}
		\begin{aligned}	\lim_{n\rightarrow \infty}\mathbb{E}\sup_{0\le s\le T}|X_s^n-X_s|^p=0\quad\text{and}\quad \mathbb{E}\sup_{0\le s\le T}|X_t|^p\leq C_p(1+\mathbb{E}|\xi|^p).
		\end{aligned}
	\end{equation}
	Then, by the properties of the Wasserstein distance,
	\begin{equation}\label{eos46}
		\begin{aligned}	
			\lim\limits_{n\rightarrow \infty}\sup_{0\le s\le T}W_p(\mu_{X_t^n}, \mu_{X_t})\le \lim\limits_{n\rightarrow \infty} \Big(\mathbb{E}\sup_{0\le s\le T}|X_s^n-X_s|^p\Big)^\frac{1}{p}=0.	
		\end{aligned}\nonumber
	\end{equation}
	By H\"older's inequality  and equation \eqref{eos18},  for any  $p>0$, we have
	\begin{align}\label{eos48}	
		&\mathbb{E}\int_0^T\big|b(s, X_s^n, \mu_{X_s^n})-b(s, X_s, \mu_{X_s})\Big|^pds\nonumber\\
		&\le C_p\Big(L_R+\mathbb{E}|X_s^n|+\mathbb{E}|X_s|\Big)^p\mathbb{E}\int_0^T\Big(|X_s^n-X_s|+\mathbb{E}|X_s^n-X_s|\Big)^pds\nonumber\\
		&\quad +C_p\mathbb{E}\Bigg[\int_0^T\Big(1+|X_s^n|+|X_s|+\mathbb{E}|X_s^n|+\mathbb{E}|X_s|\Big)^pds\mathbbm{1}_{\big\{\sup_{t\leq T}(|X_t^n|\vee|X_t|)>R\big\}}\Bigg]\nonumber\\
		&\to 0, \qquad \mathrm{by ~sending} ~~n\to\infty ~\mathrm{and ~then} ~~R\to\infty.
	\end{align}
	Similarly,
	\begin{equation}\label{eos49}
		\begin{aligned}	
			&\lim\limits_{n\rightarrow \infty}\mathbb{E}\sup_{0\le t\le T}\Big|\int_0^t\big(\sigma(s, X_s^n)-\sigma(s, X_s)\big)dB_s\Big|^p=0.
		\end{aligned}
	\end{equation}

It remains to prove the fourth item in Definition \ref{def:mvsvi}.
	Set $$\phi_t^n:=\int_0^t\nabla\psi^n(X_s^n)ds.$$ By  equations \eqref{eos45}, \eqref{eos48}, and \eqref{eos49},  $\{\phi^n\}$ is a Cauchy sequence on $\mathbb{S}^p[0, T]$ and thus there exists $\phi\in \mathbb{S}^p[0, T]$ such that 	\begin{align}\label{phin}\lim_{n\rightarrow \infty}\mathbb{E}\sup\limits_{0\le s\le T}|\phi_s^n-\phi_s|^p=0.\end{align}
	By Theorem \ref{d14} and equation \eqref{eos25}, we have
	\begin{align}\label{Jnxn}
		&\hspace{-0.5cm}\lim\limits_{n\rightarrow \infty}\mathbb{E}\sup\limits_{0\le s\le T}|J_nX_s^n-X_s|\nonumber\\
		&\le \lim\limits_{n\rightarrow \infty}\mathbb{E}\sup\limits_{0\le s\le T}|J_nX_s^n-X^n_s|+\lim\limits_{n\rightarrow \infty}\mathbb{E}\sup\limits_{0\le s\le T}|X^n_s-X_s|\nonumber\\
		&\le \lim\limits_{n\rightarrow \infty}\frac{C\Big(\mathbb{E}\sup\limits_{0\le s\le T}|\nabla\psi^n(X_s^n)|^4\Big)^\frac{1}{4}}{n}\nonumber\\
		&=0,
	\end{align}
	and then
		\begin{align}\label{Jnx}
		&\hspace{-0.5cm}\lim\limits_{n\rightarrow \infty}\mathbb{E}\sup\limits_{0\le s\le T}|J_nX_s-X_s|\nonumber\\
		&\le \lim\limits_{n\rightarrow \infty}\mathbb{E}\sup\limits_{0\le s\le T}|J_nX_s-J_nX^n_s|+\lim\limits_{n\rightarrow \infty}\mathbb{E}\sup\limits_{0\le s\le T}|J_nX^n_s-X_s|\nonumber\\
		&\le \lim\limits_{n\rightarrow \infty}\mathbb{E}\sup\limits_{0\le s\le T}|X^n_s-X_s|=0.
	\end{align}	
	
	By equations \eqref{pconX}, \eqref{phin}-\eqref{Jnx}, there exists a space  $\Omega_0$ with $\mathbb{P}(\Omega_0)=1$  satisfying that for any $\omega\in\Omega_0$, $X_t(\omega)\in \overline{D}$ for any $t\in [0, T]$, and that there exists a subsequence (still denoted as $(X^n, \phi^n)$) such that
		\begin{align*}
			\lim\limits_{n\rightarrow \infty}\sup\limits_{0\le s\le T}|X_s^n(\omega)-X_s(\omega)|=0,&\quad
			\lim\limits_{n\rightarrow \infty}\sup\limits_{0\le s\le T}|\phi_s^n(\omega)-\phi_s(\omega)|=0,\\
			\lim\limits_{n\rightarrow\infty}\sup\limits_{0\le s\le T}|J_nX_s^n(\omega)-X_s(\omega)|=0&\qquad\text{and}\quad
			\lim\limits_{n\rightarrow\infty}\sup\limits_{0\le s\le T}|J_nX_s(\omega)-X_s(\omega)|=0.
		\end{align*}
	Then by Fatou's Lemma, we have for any $\omega\in\Omega_0$,
	\begin{equation} \label{eospsi}
		\begin{aligned}
			\int_s^t\psi(X_r(\omega))dr\le \liminf\limits_{n\rightarrow
				\infty}\int_s^t\psi(J_nX_r^n(\omega))dr.
		\end{aligned}
	\end{equation}
	{Set 
\begin{align*}
&\Omega_1:=\Big\{\omega;\; \liminf_n|\phi^n(\omega)|_0^T<+\infty\Big\},\\ &\Omega_2:=\Big\{\omega; ~|\phi^n(\omega)|_0^T<+\infty,\; \phi^n_0(\omega)=0,\;\mathrm{and ~for ~all} ~\varrho\in C([0,T];\overline{D}),\\ &\hspace{3cm}(\varrho_t-X^n_t(\omega))d\phi^n_t(\omega)+\psi(X^n_t(\omega))dt\le \psi(\varrho_t)dt, ~\forall n\in\mathbb{N}\Big\}.\\
&\Omega_3:=\Big\{\omega;\ \sup_{t\leq T}|X_t(\omega)|<\infty\Big\}.
\end{align*} Then it follows from equation \eqref{eos19} that
	$
	\mathbb{P}(\Omega_0\cap\Omega_1\cap\Omega_2\cap\Omega_3)=1
	$
	and  for any $\omega\in\Omega_0\cap\Omega_1\cap\Omega_2\cap\Omega_3$, there exists a subsequence $\{n_k\}$ (depending on $\omega$ possibly) such that
	$$\lim_{k\to\infty}|\phi^{n_k}(\omega)|_0^T<\infty\quad \text{and}\quad\sup_k|\phi^{n_k}(\omega)|_0^T<\infty.$$		
	Then given any partition of $[0,T]$: $0=t_0<t_1<\cdots<t_m=T$, 
	\begin{equation} \begin{aligned}
			\sum_{i=0}^{m-1}|\phi_{t_{i+1}}(\omega)-\phi_{t_i}(\omega)|&\le
			\sum_{i=0}^{m-1}\lim_{k\to\infty}|\phi^{n_k}_{t_{i+1}}(\omega)-\phi^{n_k}_{t_i}(\omega)|\\
			&=\lim_{k\to\infty}\sum_{i=0}^{m-1}|\phi^{n_k}_{t_{i+1}}(\omega)-\phi^{n_k}_{t_i}(\omega)|\\&\le \sup_{k}|\phi^{n_k}(\omega)|_0^T<\infty,
		\end{aligned}\nonumber
	\end{equation}
	which yields $|\phi(\omega)|_0^T<\infty$.} 

{Meanwhile, for any $0\leq s\leq t\leq T$, it follows from Lemma \ref{intconv} that 
	\begin{align*}
			\lim_{k\to\infty}\int_s^t X^{n_k}_r(\omega)d\phi^{n_k}_r(\omega)= \int_s^tX_r(\omega)d\phi_r(\omega).
		\end{align*}
and moreover, 
\begin{equation}\label{intxphi}\begin{aligned}
			\lim_k\int_s^t\big[\varrho_r(\omega)-X^{n_k}_r(\omega)\big]d\phi^{n_k}(\omega)
			=\int_s^t\big[\varrho_r(\omega)-X_r(\omega)\big]d\phi(\omega).
		\end{aligned}
	\end{equation}
	Noticing that $\psi(J_nx)\le \psi^n(x)\le\psi(x)$, for any  $\varrho\in C([0, T]; \overline{D})$ and any $T\ge t>s\ge0$, we have
	\begin{equation} \label{eos50}
		\begin{aligned}
			\int_s^t(\varrho_r-X_r^{n_k}(\omega))d\phi^{n_k}_r(\omega)&\le \int_s^t\psi^{n_k}(\varrho_r)dr-\int_s^t\psi^{n_k}(X_r^{n_k}(\omega))dr\\
			&\le \int_s^t\psi(\varrho_r)dr-\int_s^t\psi(J_{n_k}X_r^{n_k}(\omega))dr.
		\end{aligned}
	\end{equation}}
		Then taking limits on both sides of equation \eqref{eos50}, 
	\begin{equation} \label{eos53}
		\begin{aligned}
			\int_s^t(\varrho_r-X_r(\omega))d\phi_r(\omega)\le \int_s^t\psi(\varrho_r)dr-\int_s^t\psi(X_r(\omega))dr.
		\end{aligned}	\nonumber
	\end{equation}
	Therefore, $(X, \phi)$ is the solution to equation \eqref{eos1}.
	\bigskip
	
	\noindent\textbf{Step $6$.} 	At last, we are going to prove the uniqueness of the solution to equation \eqref{eos1}.
	Suppose $(X, \phi)$ and $(\overline{X}, \overline{\phi})$ both are the solution to equation \eqref{eos1}. Similar to equation \eqref{eos18}, with analogous arguments we obtain that for any $p>0$,
$$
			\mathbb{E}\sup\limits_{0\le s\le T}(1+|X_s|^2)^{\frac{p}{2}}\le C_p(1+\mathbb{E}|\xi|^p).
$$
	Similarly,
$$
			\mathbb{E}\sup\limits_{0\le s\le T}(1+|\overline{X}_s|^2)^{\frac{p}{2}}\le C_p(1+\mathbb{E}|\xi|^p).
$$
	For any $R>0$, set
	\begin{equation}\label{eos63}
		\Omega_R:=\left\{\omega\in \Omega;\; \sup\limits_{0\le s\le T}|X_s|\vee \sup\limits_{0\le s\le T}|\overline{X}_s|> R\right\}.\nonumber
	\end{equation}
	Using the Yamada-Watanabe function, by It\^o's formula, we have
		\begin{align*}
			V_{\epsilon, \delta}(X_t-\overline{X}_t) &= \int_0^tV_{\epsilon, \delta}'(X_s-\overline{X}_s)\Big(b(s, X_s, \mu_{X_s})-b(s, \overline{X}_s, \mu_{\overline{X}_s})\Big)ds\\
			&\quad+\frac{1}{2}\int_0^tV_{\epsilon, \delta}^{''}(X_s-\overline{X}_s)\Big|\sigma(s, X_s)-\sigma(s, \overline{X}_s)\Big|^2ds\\
			&\quad+\int_0^tV_{\epsilon, \delta}'(X_s-\overline{X}_s)\Big(\sigma(s, X_s)-\sigma(s, \overline{X}_s)\Big)dB_s\\
			&\quad-\int_0^tV_{\epsilon, \delta}'(X_s-\overline{X}_s)d(\phi_s-\overline{\phi}_s)\\
			&=:\mathcal{J}_{6, 1}(t)+\mathcal{J}_{6, 2}(t)+\mathcal{J}_{6, 3}(t)+\mathcal{J}_{6, 4}(t).
		\end{align*}
	Then with arguments analogous to those in obtaining equation \eqref{eos42},
	\begin{equation}\label{eos69}
		\begin{aligned}
			&\mathbb{E}|X_t-\overline{X}_t|\\
			&\le  \epsilon+ \mathbb{E}[V_{\epsilon, \delta}(X_t-\overline{X}_t)]\nonumber\\
			&\le \epsilon+\int_0^t\Big(L_R+\mathbb{E}|X_s|+\mathbb{E}|\overline{X}_s|\Big)\mathbb{E}|X_s-\overline{X}_s|ds+\frac{\epsilon^{2\alpha}}{\ln \delta}+\frac{c_q^{1/2}}{R^\frac{q}{2}}+\epsilon.
		\end{aligned} \nonumber
	\end{equation}
	First sending $\epsilon\rightarrow 0$, then by Gr\"onwall's Lemma, we have
$$
			\mathbb{E}|X_t-\overline{X}_t|\le \frac{C_qe^{c_2T}(1+R)^{t}}{R^\frac{q}{2}},
$$
	where $c_2:=\mathbb{E}\sup_{t\leq T}(|X_t|+|\overline{X}_t|)$.
	Choosing  $q>2T$ and then sending $R\rightarrow 0$, we have
	\begin{equation}\label{eos71}
		\mathbb{E}|X_t-\overline{X}_t|=0.\nonumber
	\end{equation}	
	Then $X$ and $\overline{X}$ are modification of each other (i.e., $X_t=\overline{X}_t$, $\mathbb{P}$-a.s., for each $t$),
	and they are both continuous processes that are indistinguishable, i.e.,
$$
		\mathbb{P}\Big(X_t=\overline{X}_t, \;\forall 0\le t\le T\Big)=1,
$$
	and moreover
	$$	\mathbb{P}\Big(\phi_t=\overline{\phi}_t, \;\forall 0\le t\le T\Big)=1.
	$$ This proves the pathwise uniqueness of the strong solution.
\end{proof}

\section{Well-posedness of the MVFBSVS}
\label{sec:Well-posedness_MVFBSVS}
In this section, we prove the existence and uniqueness of solutions  to equation \eqref{eb1} in Theorem \ref{thb2}.
 Now we impose the following conditions on the backward equation.
\begin{assumption}\label{assumption4}
	For any $x, y, z, y_1, z_1, y_2, z_2\in\mathbb{R} $, $\mu, \mu' \in \mathcal{P}_1(\mathbb{R})$, $\nu, \nu_1,  \nu_2\in \mathcal{P}_2(\mathbb{R})$, $0<t<T$, and $\omega\in\Omega$, $F(\cdot, \cdot, x, y, z, \mu, \nu)$ is progressively measurable such that
	$
	\mathbb{E}\int_0^T|F(s, x, y, z, \mu, \nu)|^2ds<\infty
	$, and there exist constants $C>0$ and  $l>1, ~0<k<1$, such that
		$$|F(\omega, t, x, y, z, \mu, \nu)|\le C\Big(1+|x|^l+|y|^k+|z|^k+\mu(|\cdot|^l)^{\frac1l}+(\nu(|\cdot|^2))^\frac{1}{2}\Big), $$
	$$|G(\omega, x, \mu)|\le C(1+|x|+\mu(|\cdot|)),$$
	$$\big|G(\omega, x, \mu)-G(\omega, x', \mu')\big|\le C\big(|x-x'|+W_1(\mu, \mu')),$$
	and there exists a constant $L_R>0$  satisfying that $e^{L_R^2}\leq C(1+R)$ for any $R>0$, such that if $|x|\leq R$,
	\begin{align*}
		%\label{eqn:assumption4_eqn1}
		&\hspace{-0.3cm}\Big|F(\omega, t, x, y_1, z_1, \mu, \nu_1)-F(\omega, t, x, y_2, z_2, \mu, \nu_2)\Big|\\
		&\le \big(L_R+\mu(|\cdot|^l)^{\frac1l}\big)\Big(|y_1-y_2|+|z_1-z_2|+W_2(\nu_1, \nu_2)\Big).\nonumber
	\end{align*}
	Furthermore, $\psi_2$  is  {convex} and lower semicontinuous satisfying that $\psi_2(x)\ge \psi_2(0)=0$, $0\in \operatorname{Int}(D_2)$ where $D_2:=\{x:\partial\psi_2(x)\neq\emptyset\}$, and there exist constants $l>1$ and $C>0$ such that
	$$	\psi_2(x)\le C(1+|x|^l)\quad\text{and}\quad  {\mathbb{E}\Big[|\xi|^{4\vee (2l)}+|\psi(\xi)|^2\Big]<\infty}.$$
\end{assumption}
We first give the definition of the solution to equation \eqref{eb1}.
\begin{definition}
	\label{eqn:def_MVFBSVS}
	A quintuple of {progressively measurable} processes $(X, Y, Z, \phi, \phi^{(2)})$ defined on  $(\Omega, \mathscr{F}, \{\mathscr{F}_t\}_{t\geq0}, \mathbb{P})$ is called a solution to equation \eqref{eb1} if it satisfies the following conditions:
	\begin{itemize}
		\setlength\itemsep{0.18em}
\item $(X, \phi)$ is a solution of equation \eqref{eos1}.
		\item $(Y, Z)\in \mathbb{S}^2[0,T]\times \mathbb{H}^2[0,T]$, and $Y_t\in \overline{D}_2$ for any $t\in [0,T]$, $\mathbb{P}$-a.s..
		\item  $\phi^{(2)}$ is a continuous process with bounded variation on $[0,T]$.
		\item $(Y, Z, \phi^{(2)})$ satisfies that for any $0\leq t\leq T$,
		$$
		\begin{aligned}
			Y_t=&G(X_T, \mu_{X_T})+\int_t^TF(s, X_s, Y_s, Z_s, \mu_{X_s}, \mu_{Y_s})ds-\int_t^TZ_sdB_s\\
			&\hspace{7cm}-(\phi^{(2)}_T-\phi^{(2)}_t), \quad \mathbb{P}-a.s..
		\end{aligned}
		$$
		\item 	For any $\varrho_2 \in C([0, T]; \overline{D}_2)$ and $0\leq s<t\leq T$,
		$$
		\begin{aligned}
			\int_s^t(\varrho_2(u)-Y_u)d\phi^{(2)}_u+\int_s^t\psi_2(Y_u)du\le \int_s^t\psi_2(\varrho_2(u))du, \quad \mathbb{P}-a.s..
		\end{aligned}
		$$	
	\end{itemize}
\end{definition}

%The existence and uniqueness theorem for solutions of classical distribution-dependent backward stochastic differential equations, and the main results are presented in the following theorem \cite{carmona2018probabilistic}.
%\begin{theorem}\label{eb999}
%The following equation
%	$$
%	dY_t=-F(t, Y_t, Z_t, \mu_{Y_t})dt+ZdB_t, \qquad Y_T=\xi\in L^2(\Omega).
%	$$
%	Suppose for all $y, y', z, z'\in \mathbb{R}$, $\omega\in\Omega$, $\mu, \mu' \in\mathcal{P}_2(\mathbb{R})$, the function $F(\omega,t,y,z,\mu)$ is progressive measurable w.r.t. $(\omega,t)$ and belongs to $\mathbb{H}^2[0,T]$, and
%	$$
%	|F(t, y, z, \mu)-F(t, y', z', \mu')|\le C(|y-y'|+|z-z'|+W_2(\mu, \mu')).
%	$$
%	Then there exists a unique solution $(Y,Z)$ to the equation satisfying that $(Y, Z) \in \mathbb{S}^2[0, T]\times\mathbb{H}^2[0, T]$.
%\end{theorem}

Theorem \ref{thos1}  establishes the existence and uniqueness of solution $(X,\phi)$ for the forward equation in equation \eqref{eos1} and $X\in \mathbb{S}^p[0,T]$ for any $p\geq2$. To prove the existence and uniqueness of the solution for equation \eqref{eb1}, we consider the following approximating system of equations:
\begin{equation} \label{eb2}
	\left \{\begin{aligned}
		dY^n_t = &-F^n(t, X_t, Y^n_t, Z^n_t, \mu_{X_t}, \mu_{Y^n_t})dt+Z^n_tdB_t+\nabla\psi_2^n(Y^n_t)dt, \\
		Y^n_T=&G(X_T, \mu_{X_T}).
	\end{aligned} \right.
\end{equation}	
Here, $\psi_2^n(x)$ is the Yosida-Moreau approximation function, and
$$F^n( t, X_t, y, z, \mu_{X_t}, \nu):=F( t, X_t, y, z, \mu_{X_t}, \nu)\mathbbm{1}_{\{\sup_{s\leq t}|X_s|\leq n\}}.$$  Then for every $n\geq1$,
\begin{align*}
&\hspace{-0.8cm}\Big|F^n(t,X_t, y_1, z_1, \mu_{X_t},\nu_1)-F^n(t,X_t, y_2, z_2, \mu_{X_t},\nu_2)\Big|\\
&\leq \big(L_n+\mathbb{E}|X_t|^l\big)\Big(|y_1-y_2|+|z_1-z_2|+W_1(\nu_1,\nu_2)\Big).
\end{align*}
By Theorem $4.24$ in \cite{carmona2018probabilistic}, for every $n\geq1$, equation \eqref{eb2} has a unique solution $(Y^n, Z^n)\in\mathbb{S}^2[0, T]\times\mathbb{H}^2[0, T]$.

\begin{lemma}\label{lb1}
	Under the Assumptions \ref{assumption1}, \ref{assumption2},  and \ref{assumption4}, for $Y^n$ and $Z^n$ evolving according to
	equation \eqref{eb2},
	there exists a constant $C_1>0$ depending on $T$ and $\mathbb{E}|\xi|^{2l}$ such that
	$$\sup\limits_{n}\left[\mathbb{E}\sup\limits_{0\le s\le T}|Y_s^n|^2+\mathbb{E}\int_0^T|Z_s^n|^2ds\right]\le C_1.$$		
\end{lemma}	
\begin{proof}
	By It\^o's formula, for any $\lambda>0$,
	\begin{align}\label{eb3}
		&e^{\lambda t}|Y_t^n|^2+\lambda\int_t^Te^{\lambda s}|Y_s^n|^2ds+\int_t^Te^{\lambda s}|Z_s^n|^2ds+2\int_t^Te^{\lambda s}Y_s^n\nabla\psi_2^n(Y_s^n)ds\nonumber\\
		&={e^{\lambda T}}|G(X_T, \mu_{X_T})|^2+2\int_t^Te^{\lambda s}Y_s^nF^n(s, X_s, Y_s^n, Z_s^n, \mu_{X_s}, \mu_{Y_s^n})ds\nonumber\\
		&\hspace{6cm}-2\int_t^Te^{\lambda s}Y_s^nZ_s^ndB_s.
	\end{align}
	Taking expectations, we have
		\begin{align*}
			&\mathbb{E}e^{\lambda t}|Y_t^n|^2+\lambda \mathbb{E}\int_t^Te^{\lambda s}|Y_s^n|^2ds+\mathbb{E}\int_t^Te^{\lambda s}|Z_s^n|^2ds+2\mathbb{E}\int_t^Te^{\lambda s}Y_s^n\nabla\psi_2^n(Y_s^n)ds\\
			&={e^{\lambda T}}\mathbb{E}|G(X_T, \mu_{X_T})|^2+2\mathbb{E}\int_t^Te^{\lambda s}Y_s^nF^n(s, X_s, Y_s^n, Z_s^n, \mu_{X_s}, \mu_{Y_s^n})ds.
		\end{align*}
	Note that by Young's inequality and Theorem \ref{thos1},
	\begin{equation} \label{eb5}
		\begin{aligned}
			&2\mathbb{E}\int_t^Te^{\lambda s}Y_s^nF^n(s, X_s, Y_s^n, Z_s^n, \mu_{X_s}, \mu_{Y_s^n})ds	\\
			&\le 2C\mathbb{E}\int_t^Te^{\lambda s}|Y_s^n|\Big(1+|X_s|^l+|Y_s^n|^k+|Z_s^n|^k+(\mathbb{E}|X_s|^l)^{1/l}+\sqrt{\mathbb{E}|Y_s^n|^2}\Big)ds\\
			&\le \frac{C'}{\epsilon}\mathbb{E}\int_t^Te^{\lambda s}|Y_s^n|^2ds\\
			&\quad+4C'\epsilon \mathbb{E}\int_t^Te^{\lambda s}\Big(1+|X_s|^{2l}+|Y_s^n|^2+|Z_s^n|^2+\mathbb{E}|X_s|^{2l}+\mathbb{E}|Y_s^n|^2\Big)ds.
		\end{aligned}\nonumber
	\end{equation}
	By the monotonicity of $\nabla\psi_2^n(x)$, we know that $$\mathbb{E}\int_t^Te^{\lambda s}Y_s^n\nabla\psi_2^n(Y_s^n)ds\ge 0.$$
	By Assumption \ref{assumption4} and Theorem \ref{thos1}, we have
	$$
	\mathbb{E}|G(X_T, \mu_{X_T})|^2\le C_T(1+\mathbb{E}|\xi|^2).
	$$
	Choosing $\epsilon=\frac{1}{16C'}$ and $\lambda=\lambda_1:=16C'^2+1$, it follows from Theorem \ref{thos1} that
		\begin{align}\label{eb6}
			&\hspace{-0.5cm}\frac{1}{2}\mathbb{E}\int_0^Te^{\lambda_1 s}|Y_s^n|^2ds+\frac{1}{2}\mathbb{E}\int_0^Te^{\lambda_1 s}|Z_s^n|^2ds\nonumber\\
			&\le {e^{\lambda_1 T}}\mathbb{E}|G(X_T, \mu_{X_T})|^2+\frac{1}{4}\mathbb{E}\int_0^Te^{\lambda_1 s}\Big(1+|X_s|^{2l}+\mathbb{E}|X_s|^{2l}\Big)ds\nonumber\\
			&\le C_{T}(1+\mathbb{E}|\xi|^{2l}).
		\end{align}
	By the BDG inequality, we obtain
	%\begin{equation} \label{eb7}
	\begin{align*}
		\mathbb{E}\sup\limits_{0\le t\le T}\left|\int_t^TY_s^nZ_s^ndB_s\right|\le \frac{1}{2}\mathbb{E}\sup\limits_{0\le s\le  T}|Y_s^n|^2+C\mathbb{E}\int_0^T|Z_s^n|^2ds.
	\end{align*}
	%	\end{equation}
Plugging the above inequality into equation \eqref{eb3}, we have
\begin{equation} \label{eb8}
	\begin{aligned}
		\sup_n\mathbb{E}\sup\limits_{0\le s\le T}|Y_s^n|^2\le C_T(1+\mathbb{E}|\xi|^{2l}).
	\end{aligned}
\end{equation}
Combining equations \eqref{eb6} and \eqref{eb8}, we have
\begin{equation} \label{eb9}
	\begin{aligned}
		\sup\limits_{n}\left[\mathbb{E}\sup\limits_{0\le s\le T}|Y_s^n|^2+\mathbb{E}\int_0^T|Z_s^n|^2ds\right]\le C_T(1+\mathbb{E}|\xi|^{2l}).
	\end{aligned}\nonumber
\end{equation}
\end{proof}
\begin{lemma}\label{lb2}
Under Assumptions \ref{assumption1}, \ref{assumption2} and  \ref{assumption4}, for $Y^n$ evolving according to
equation \eqref{eb1}, there exists $C' > 0$ depending on $T$ and $\mathbb{E}|\xi|^{2l}$ such that
$$\sup\limits_{n}\int_0^T\mathbb{E}|\nabla\psi_2^n(Y_s^n)|^2ds\le C'.$$	
\end{lemma}	
\begin{proof}
By {It\^o's formula and the convexity of $\psi_2^n$}, for any $ 0<s<t\le T$ and any $\lambda>0$,
	\begin{align*}
e^{\lambda t}\psi_2^n(Y_t^n)-e^{\lambda s}\psi_2^n(Y_s^n)\geq
\int_s^t\lambda e^{\lambda r}\psi_2^n(Y_r^n)dr
+\int_s^te^{\lambda r}\nabla\psi_2^n(Y_r^n)dY_r^n.
\end{align*}
We then have
%	\begin{equation} \label{eb10}
	\begin{align*}
		&\hspace{-0.8cm}e^{\lambda t}\psi_2^n(Y_t^n)+\lambda\int_t^Te^{\lambda s}\psi_2^n(Y_s^n)ds+\int_t^Te^{\lambda s}|\nabla\psi_2^n(Y_s^n)|^2ds\\
	&\le e^{\lambda T}\psi_2^n(G(X_T, \mu_{X_T}))-\int_t^Te^{\lambda s}\nabla\psi_2^n(Y_s^n)Z_s^ndB_s\\
		&\quad+\int_t^Te^{\lambda s}\nabla\psi_2^n(Y_s^n)F^n(t, X_s, Y_s^n, Z_s^n, \mu_{X_s}, \mu_{Y_s^n})ds.
	\end{align*}
	%	\end{equation}
Note that $\int_t^T\nabla\psi_2^n(Y_s^n)Z_s^ndB_s$ is a martingale by Lemma \ref{lb1} and properties of $\nabla\psi_2^n$. We have
	\begin{align}\label{eb11}
		&e^{\lambda t}\mathbb{E}\psi_2^n(Y_t)+\lambda\mathbb{E}\int_t^Te^{\lambda s}\psi_2^n(Y_s^n)ds+\mathbb{E}\int_t^Te^{\lambda s}|\nabla\psi_2^n(Y_s^n)|^2ds\\
		&=e^{\lambda T}\mathbb{E}\psi_2^n(G(X_T, \mu_{X_T}))+\mathbb{E}\int_t^Te^{\lambda s}\nabla\psi_2^n(Y_s^n)F^n(s, X_s, Y_s^n, Z_s^n, \mu_{X_s}, \mu_{Y_s^n})ds.\nonumber
	\end{align}
By Assumption \ref{assumption4} and Theorem \ref{thos1}, we have
\begin{equation*}%\label{eb13}
	\sup\limits_{n}\mathbb{E}\psi_2^n(G(X_T, \mu_{X_T}))\le C_T(1+\mathbb{E}|\xi|^l),
\end{equation*}
and by Lemma \ref{lb1} we have
%\begin{equation} \label{eb12}
\begin{align*}
&\hspace{-1.5cm}\mathbb{E}\int_t^T\nabla\psi_2^n(Y_s^n)F^n(s, X_s, Y_s^n, Z_s^n, \mu_{X_s}, \mu_{Y_s^n})ds\\
	&\le \frac{1}{2}\mathbb{E}\int_t^T|\nabla\psi_2^n(Y_s^n)|^2ds+C(1+\mathbb{E}|\xi|^{2l}).
\end{align*}
%\end{equation}	
 Substituting the above two inequalities into equation \eqref{eb11}, and noting that  $\psi_2^n(x)\geq 0$, we obtain
\begin{equation} \label{eb14}
	\begin{aligned}
		\sup\limits_{n}\mathbb{E}\int_0^T|\nabla\psi_2^n(Y_s^n)|^2ds\le C_T(1+\mathbb{E}|\xi|^{2l}).
	\end{aligned}\nonumber
\end{equation}
\end{proof}
\begin{lemma}\label{lb3}
Suppose Assumptions \ref{assumption1}, \ref{assumption2} and  \ref{assumption4} hold. When $0<T\le T_0$ for some $T_0>0$, the sequence $(Y^n, Z^n)_n$ is a Cauchy sequence on $\mathbb{S}^2[0, T]\times\mathbb{H}^2[0, T]$.		
\end{lemma}	
\begin{proof}
Without loss of generality, let $m\geq n$, and denote $$Y_t^{n, m}=Y_t^n-Y_t^m\quad \text{and} \quad Z_t^{n, m}=Z_t^n-Z_t^m.$$ Then, by It\^o's formula, for $\lambda > 0$,
%\begin{equation}
\begin{align}
	&e^{\lambda t}|Y_t^{n, m}|^2+\lambda\int_t^Te^{\lambda s}|Y_s^{n, m}|^2ds+\int_t^Te^{\lambda s}|Z_s^{n, m}|^2ds\nonumber\\
	&+2\int_t^Te^{\lambda s}Y_s^{n, m}(\nabla\psi_2^n(Y_s^n)-\nabla\psi_2^m(Y_s^m))ds+2\int_t^Te^{\lambda s}Y_s^{n, m}Z_s^{n, m}dB_s\nonumber\\
	&=2\int_t^Te^{\lambda s}Y_s^{n, m}\Big(F^n(s, X_s, Y_s^n, Z_s^n, \mu_{X_s}, \mu_{Y_s^n})\\
	&\hspace{4cm}-F^m(s, X_s, Y_s^m, Z_s^m, \mu_{X_s}, \mu_{Y_s^m})\Big)ds\nonumber.\label{eb15}
\end{align}
%	\end{equation}	
By Lemma \ref{lb2} and the properties of the Yosida-Moreau function (Theorem \ref{d14}),
\begin{equation} \label{eb19}
\begin{aligned}
	&\hspace{-1cm}-2\mathbb{E}\int_t^Te^{\lambda s}Y_s^{n, m}(\nabla\psi_2^n(Y_s^n)-\nabla\psi_2^m(Y_s^m))ds\\
	&\le 2\mathbb{E}\int_t^Te^{\lambda s}\left(\frac{1}{n}+\frac{1}{m}\right)\nabla\psi_2^n(Y_s^n)\nabla\psi_2^m(Y_s^m)ds\\
	&\le 2\left(\frac{1}{n}+\frac{1}{m}\right)e^{\lambda T}\sqrt{\mathbb{E}\int_t^T|\nabla\psi_2^n(Y_s^n)|^2ds}\sqrt{\mathbb{E}\int_t^T|\nabla\psi_2^m(Y_s^m)|^2ds}\\
	&\le2C'e^{\lambda T}\left(\frac{1}{n}+\frac{1}{m}\right).
\end{aligned}\nonumber
\end{equation}	
Denoting $\Lambda:=\Big\{\sup\limits_{0\le t\le T}|X_t|> n\Big\}$ and for $p\geq1$, 
\begin{align}\label{eqn:setA}
c_{p}:=\big(\mathbb{E}\sup_{s\leq T}|X_s|^p\big)^{1/p},\quad\quad a_2:=\sup_n\big(\mathbb{E}\sup_{s\leq T}|Y^n_s|^2\big)^{1/2}.
\end{align}
we have
\begin{align*}
&2\left|\int_t^Te^{\lambda s}Y_s^{n, m}\Big(F^n(s, X_s, Y_s^n, Z_s^n, \mu_{X_s}, \mu_{Y_s^n})-F^m(s, X_s, Y_s^m, Z_s^m, \mu_{X_s}, \mu_{Y_s^m})\Big)ds\right|\\
&\le2\int_t^Te^{\lambda s}|Y_s^{n, m}|\Big(F^n(s, X_s, Y_s^n, Z_s^n, \mu_{X_s}, \mu_{Y_s^n})\\
&\hspace{5cm}-F^n(s, X_s, Y_s^m, Z_s^m, \mu_{X_s}, \mu_{Y_s^m})\Big)ds{\mathbbm{1}_{\Lambda^c}}\\
&\quad+4\int_t^Te^{\lambda s}|Y_s^{n, m}|\Big|F(s, X_s, Y_s^m, Z_s^m, \mu_{X_s}, \mu_{Y_s^m})\Big|ds{\mathbbm{1}_{\Lambda}}\\
&\le 2\int_t^Te^{\lambda s}\Big(L_n+c_{l}\Big)|Y_s^{n, m}|\Big(|Y_s^{n, m}|+|Z_s^{n, m}|+\sqrt{\mathbb{E}|Y_s^{n, m}|^2}\Big)ds{\mathbbm{1}_{\Lambda^c}}\\
&\quad+C\int_t^Te^{\lambda s}|Y_s^{n, m}|\Big(1+2c_{l}+2a_2+|Y_s^m|^k+|Z_s^m|^k+|Y_s^n|^k+|Z_s^n|^k\Big)ds{\mathbbm{1}_{\Lambda}}\\
&\le \frac{1}{\epsilon}\int_t^Te^{\lambda s}|Y_s^{n, m}|^2ds\\
&\quad+6\epsilon (L_n+c_{l})^2\int_t^Te^{\lambda s}\Big(|Y_s^{n, m}|^2+|Z_s^{n, m}|^2+\mathbb{E}|Y_s^{n, m}|^2\Big)ds{\mathbbm{1}_{\Lambda^c}}\\
&\quad+C^2\epsilon \int_t^Te^{\lambda s}\Big(1+2c_{l}+2a_2+|Y_s^m|^k+|Z_s^m|^k+|Y_s^n|^k+|Z_s^n|^k\Big)^2ds{\mathbbm{1}_{\Lambda}}.
\end{align*}
Choosing $\epsilon=\frac1{12(L_n+c_{l})^2}$ and  $\lambda=12(L_n+c_{l})^2+3/2$, we obtain by applying H\"older's inequality, Chebyshev's inequality, Theorem \ref{thos1}, and Lemma \ref{lb1} that 
\begin{align}\label{eb20}
	&\hspace{-1.5cm}\frac12\mathbb{E}\left[\int_0^Te^{\lambda s}\big(|Y_s^{n, m}|^2+|Z_s^{n, m}|^2\big)ds\right]\nonumber\\
	&	\le C'e^{\lambda T }\left(\frac{1}{n}+\frac{1}{m}\right)+Ce^{\lambda T}\frac{c_q\epsilon}{\lambda n^{q/2}}\nonumber\\
	&\le C(1+n)^{24T}\left(\frac{1}{n}+\frac{1}{m}+c_qn^{-q/2}(L_n+c_{l})^{-4}\right),
\end{align}
which tends to $0$ by taking $q=2$ and then sending $n\to\infty$, when $T< \frac1{24}$.

Applying the BDG inequality, we have
\begin{align}\label{eb21}
\mathbb{E}\sup\limits_{0\le s\le T}|Y_{s}^{n, m}|^2
&\le C_{T}(1+n)^{24T}\left(\frac{1}{n}+\frac{1}{m}+c_qn^{-q/2}(L_n+c_{l})^{-4}\right)\nonumber\\
&\quad+\frac12\mathbb{E}\sup\limits_{0\le s\le T}|Y_s^{n, m}|^2+C\mathbb{E}\int_0^{T}|Z_s^{n, m}|^2ds,
\end{align}
and thus 
\begin{align}\label{eb211}
\mathbb{E}\sup\limits_{0\le s\le T}|Y_{s}^{n, m}|^2
\to 0, \quad \mbox{as} ~~n, m\to\infty.
\end{align}
Hence,  $\{(Y^n, Z^n)\}_{n\geq1}$ is a Cauchy sequence in $\mathbb{S}^2[0, T]\times\mathbb{H}^2[0, T]$.
\end{proof}

The following theorem is the main result of this section.
\begin{theorem}\label{thb2}
Under Assumptions \ref{assumption1}, \ref{assumption2} and  \ref{assumption4},  there exists a unique solution $( Y, Z, \phi^{(2)})$ to equation \eqref{eb1}.
\end{theorem}
\begin{proof}%[Proof of Theorem \ref{thb2}]
We first prove the existence and then the uniqueness.
\medskip

\noindent\textbf{1). Existence.}
 Since  $(Y^n, Z^n)_n$ is a Cauchy sequence in $\mathbb{S}^2[0,T] \times \mathbb{H}^2[0,T]$ for $0\le T\le T_0$ with some $T_0>0$,  there exists a pair of processes $(Y,Z)\in\mathbb{S}^2[0,T] \times \mathbb{H}^2[0,T]$  such that
\begin{equation} \label{eb23}
\begin{aligned}
	\lim\limits_{n\rightarrow\infty}\left[\mathbb{E}\sup\limits_{0\le s\le T}|Y_s^n-Y_s|^2+\mathbb{E}\int_0^T|Z_s^n-Z_s|^2ds\right]=0,
\end{aligned}\nonumber	
\end{equation}
and moreover,
\begin{equation} \label{eb25}
\begin{aligned}
	\lim\limits_{n\rightarrow\infty}\mathbb{E}\sup_{t\in[0,T]}\left|\int_t^T\big(Z_s^n-Z_s\big)dB_s\right|^2=0.
\end{aligned}	\nonumber
\end{equation}
At the same time, according to equations \eqref{eb21}, we have
\begin{align*}
	&\lim\limits_{n\rightarrow\infty}\mathbb{E}\int_0^T\Big[F^n(s, X_s, Y_s^n, Z_s^n, \mu_{X_s}, \mu_{Y_s^n})-F(s, X_s, Y_s, Z_s, \mu_{X_s}, \mu_{Y_s})\Big]^2ds\\
	&\le2\lim\limits_{n\rightarrow\infty}\mathbb{E}\int_0^T\Big[F^n(s, X_s, Y_s^n, Z_s^n, \mu_{X_s}, \mu_{Y_s^n})-F^n(s, X_s, Y_s, Z_s, \mu_{X_s}, \mu_{Y_s})\Big]^2ds\\
	&\quad+2\lim\limits_{n\rightarrow\infty}\mathbb{E}\int_0^T\Big[F^n(s, X_s, Y_s, Z_s, \mu_{X_s}, \mu_{Y_s})-F(s, X_s, Y_s, Z_s, \mu_{X_s}, \mu_{Y_s})\Big]^2ds\\
	&\le2\lim\limits_{n\rightarrow\infty}\int_0^T\big[L_n+\mathbb{E}|X_s|^l\big]^2\Big[2\mathbb{E}|Y^n_s-Y_s|^2+\mathbb{E}|Z^n_s-Z_s|^2\Big]ds\\
	&\quad+C\lim\limits_{n\rightarrow\infty}\mathbb{E}\left(\int_0^T\Big[1+|X_s|^{2l}+|Y_s|^{2k}+|Z_s|^{2k}+ (\mathbb{E}|X_s|^l)^2+\mathbb{E}|Y_s|^2\Big]ds\mathbbm{1}_{\mathcal{A}}\right)\\
	&=0,
\end{align*}
where $\mathcal{A}$ is defined in equation \eqref{eqn:setA}.
Set $\phi_t^{(2,n)}=\int_0^t\nabla\psi_2^n(Y_s^n)ds.$ Then, by equation \eqref{eb2} and the BDG inequality, we have
\begin{equation} \label{eb28}
\begin{aligned}
	\mathbb{E}\sup\limits_{0\le t\le T}|\phi_t^{(2,n)}-\phi_t^{(2,m)}|^2=0.
\end{aligned}	\nonumber
\end{equation}
Therefore, there exists $\phi^{(2)}\in \mathbb{S}^2[0, T]$ such that
\begin{equation} \label{eb29}
\begin{aligned}
	\mathbb{E}\sup\limits_{0\le t\le T}|\phi_t^{(2,n)}-\phi_t^{(2)}|^2=0.
\end{aligned}	\nonumber
\end{equation}

Noting that $\psi_2(J_nx)\le \psi^n_2(x)\le\psi_2(x)$, for any $\varrho\in C([0, T]; \overline{D}_2)$ and $t>s\ge0$, we have
\begin{equation} \label{eb0011}
\begin{aligned}
	\int_s^t(\varrho_r-Y_r^n)d\phi^{(2,n)}_r&\le \int_s^t\psi^n_2(\varrho_r)dr-\int_s^t\psi^n_2(Y_r^n)dr\\
	&\le \int_s^t\psi_2(\varrho_r)dr-\int_s^t\psi_2(J_nY_r^n)dr.
\end{aligned}
\end{equation}

From Definition \ref{d14} and Lemma \ref{lb2}, we have
\begin{align*}%\label{eJY1}
&\hspace{-0.5cm}\lim\limits_{n\rightarrow \infty}\mathbb{E}\int_0^T|J_nY_s^n-Y_s|^2ds\nonumber\\
&\le \lim\limits_{n\rightarrow \infty}\mathbb{E}\int_0^T|J_nY_s^n-Y^n_s|^2ds+\lim\limits_{n\rightarrow \infty}\mathbb{E}\int_0^T|Y^n_s-Y_s|^2ds\nonumber\\
&\le \lim\limits_{n\rightarrow \infty}\frac{C\mathbb{E}\int_0^T|\nabla\psi^n_2(Y_s^n)|^2ds}{n^2}\nonumber\\
&=0,\\
&\hspace{-1cm}\text{and}\quad\lim\limits_{n\rightarrow \infty}\mathbb{E}\int_0^T|J_nY_s-Y_s|^2ds\nonumber\\
&\le \lim\limits_{n\rightarrow \infty}\mathbb{E}\int_0^T|J_nY_s-J_nY^n_s|^2ds+\lim\limits_{n\rightarrow \infty}\mathbb{E}\int_0^T|J_nY^n_s-Y_s|^2ds\nonumber\\
&\le \lim\limits_{n\rightarrow \infty}\mathbb{E}\int_0^T|Y^n_s-Y_s|^2ds\nonumber\\
&=0.\nonumber
\end{align*}
Therefore, by the above equations and Lemmas \ref{lb1} and \ref{lb2}, similar to equation \eqref{intxphi}, there exists a space $\Omega'$ with probability 1 such that for every $\omega\in\Omega'$,  there exists a subsequence $n_k$ (which may depend on $\omega$), satisfying that  for all $\rho\in C([0, T]; \overline{D}_2)$, 
\begin{equation} \label{eb0012}
\begin{aligned}
	&\int_t^T(\rho_s-Y_s^{{n_k}}(\omega))d\phi^{(2,{n_k})}_s(\omega)\rightarrow\int_t^T(\rho_s-Y_s(\omega))d\phi^{(2)}_s(\omega),\\
&\lim\limits_{k\rightarrow\infty}\int_0^T|J_{n_k}Y_s^{n_k}(\omega)-Y_s(\omega)|^2ds=0, \\
&\lim\limits_{k\rightarrow\infty}\int_0^T|J_{n_k}Y_s(\omega)-Y_s(\omega)|^2ds=0.
\end{aligned}
\end{equation}
Then by Fatou's lemma and equations \eqref{eb0011} and \eqref{eb0012}, 
\begin{equation} \label{eb0013}
\begin{aligned}
	\int_s^t\psi_2(Y_r(\omega))dr\le \liminf\limits_{n\rightarrow
		\infty}\int_s^t\psi_2(J_{n_k}Y_r^{n_k}(\omega))dr, 
\end{aligned}	\nonumber
\end{equation}
and moreover taking limits on both sides of \eqref{eb0011} yields that
\begin{equation} \label{eb0014}
\begin{aligned}
	\int_s^t(\varrho_r-Y_r(\omega))d\phi^{(2)}_r\le \int_s^t\psi_2(\varrho_r)dr-\int_s^t\psi_2(Y_r(\omega))dr.
\end{aligned}
\end{equation}
Combining the above discussions, $(X, Y, Z, \phi, \phi^{(2)})$ is a solution to equation \eqref{eb1}.

We can repeat the above arguments on the interval $[T, 2T ]$, and iterate up to any given finite time interval
$[0, T']$ on which we obtain a solution $(X,Y, Z, \phi, \phi^{(2)})$ for equation \eqref{eb1}. 
\bigskip

\noindent\textbf{2). Uniqueness.} 
Consider $(X, \phi)$ as the unique strong solution to the forward equation \eqref{eb1}.
Suppose $(Y, Z, \phi^{(2)})$ and $(\overline{Y}, \overline{Z}, \overline{\phi}^{(2)})$ are both strong solutions to the backward equation of \eqref{eb1}. Thus,
\begin{align*}
	\mathbb{E}\sup\limits_{0\le s\le T}|Y_s|^2+\mathbb{E}\sup\limits_{0\le s\le T}|\overline{Y}_s|^2+\mathbb{E}\int_0^T|Z_s|^2ds+\mathbb{E}\int_0^T|\overline{Z}_s|^2ds\le C_T.
\end{align*}
Using It\^o's formula, for any $\lambda>0$, we have
\begin{equation} \label{eb35}
\begin{aligned}
	&e^{\lambda t}|Y_t-\overline{Y}_t|^2+\lambda\int_t^Te^{\lambda s}|Y_s-\overline{Y}_s|^2ds+\int_t^Te^{\lambda s}|Z_s-\overline{Z}_s|^2ds\\
	&+2\int_t^Te^{\lambda s}(Y_s-\overline{Y}_s)d(\phi_s^{(2)}-\overline{\phi}_s^{(2)})\\
	&=2\int_t^Te^{\lambda s}(Y_s-\overline{Y}_s)\Big(F(s, X_s, Y_s, Z_s, \mu_{X_s}, \mu_{Y_s})\\
	&\hspace{5cm}-F(s, X_s, \overline{Y}_s, \overline{Z}_s, \mu_{X_s}, \mu_{\overline{Y}_s})\Big)ds\\
	&\quad-2\int_t^Te^{\lambda s}(Y_s-\overline{Y}_s)(Z_s-\overline{Z}_s)dB_s.
\end{aligned}\nonumber
\end{equation}	
Taking expectations on both sides of the above equation, we obtain
%\begin{equation} \label{eb36}
\begin{align*}
&\mathbb{E}e^{\lambda t}|Y_t-\overline{Y}_t|^2+\lambda \mathbb{E}\int_t^Te^{\lambda s}|Y_s-\overline{Y}_s|^2ds+\mathbb{E}\int_t^Te^{\lambda s}|Z_s-\overline{Z}_s|^2ds\\
&+2\mathbb{E}\int_t^Te^{\lambda s}(Y_s-\overline{Y}_s)d(\phi_s^{(2)}-\overline{\phi}_s^{(2)})\\
&=2\mathbb{E}\int_t^Te^{\lambda s}(Y_s-\overline{Y}_s)\Big(F(s, X_s, Y_s, Z_s, \mu_{X_s}, \mu_{Y_s})\\
&\hspace{5cm}-F(s, X_s, \overline{Y}_s, \overline{Z}_s, \mu_{X_s}, \mu_{\overline{Y}_s})\Big)ds.
\end{align*}%\nonumber
%\end{equation}
According to the definition of the solution, we know that
\begin{equation} \label{eb37}
\begin{aligned}
	\mathbb{E}\int_t^Te^{\lambda s}(Y_s-\overline{Y}_s)d(\phi_s^{(2)}-\overline{\phi}_s^{(2)})\ge 0.		
\end{aligned}
\end{equation}
Similar to the proof of Lemma \ref{lb3}, we have
\begin{align*}
	&\int_t^Te^{\lambda s}|Y_s-\overline{Y}_s|\Big|F(s, X_s, Y_s, Z_s, \mu_{X_s}, \mu_{Y_s})-F(s, X_s, \overline{Y}_s, \overline{Z}_s, \mu_{X_s}, \mu_{\overline{Y}_s})\Big|ds\\
	&\le\int_t^Te^{\lambda s}	|Y_s-\overline{Y}_s|\Big[F(s, X_s, Y_s, Z_s, \mu_{X_s}, \mu_{Y_s})-F(s, X_s, \overline{Y}_s, \overline{Z}_s, \mu_{X_s}, \mu_{\overline{Y}_s})\Big]ds
	\mathbbm{1}_{\mathcal{A}^c}\\
	&\quad+\int_t^Te^{\lambda s}|Y_s-\overline{Y}_s|	\Big[F(s, X_s, Y_s, Z_s, \mu_{X_s}, \mu_{Y_s})-F(s, X_s, \overline{Y}_s, \overline{Z}_s, \mu_{X_s}, \mu_{\overline{Y}_s})\Big]ds\mathbbm{1}_{\mathcal{A}}\\
	&\le \frac{1}{\epsilon}\int_t^Te^{\lambda s}|Y_s-\overline{Y}_s|^2\\
	&\quad
	+6\epsilon (L_n+c'_{l})^2\int_t^Te^{\lambda s}\Big(|Y_s-\overline{Y}_s|^2+|Z_s-\overline{Z}_s|^2+\mathbb{E}|Y_s-\overline{Y}_s|^2\Big)ds\\
	&\quad+2C^2\epsilon \int_t^Te^{\lambda s}\Big(1+2|X_s|^{l}+2\mathbb{E}|X_s|^l+|Y_s|^k+|\overline{Y}_s|^k\\
	&\hspace{4cm}+|Z_s|^k+|\overline{Z}_s|^k+\big(\mathbb{E}|Y_s|^2\big)^{1/2}+\big(\mathbb{E}|\overline{Y}_s|^2\big)^{1/2}\Big)^2ds\mathbbm{1}_{\mathcal{A}},
\end{align*}
where $c'_{l}:=\Big(\mathbb{E}\big(\sup_{t\leq T}|X_t|^l\big)\Big)^{1/l}$.

Then similar to equation \eqref{eb20}, choosing $\epsilon=\frac1{12(L_n+c'_{l})^2}$ and  $\lambda=12(L_n+c'_{l})^2+3/2$, we obtain
\begin{equation} \label{eb39}
\begin{aligned}
	&\mathbb{E}e^{\lambda t}|Y_t-\overline{Y}_t|^2+\frac12\mathbb{E}\int_t^Te^{\lambda s}\big(|Y_s-\overline{Y}_s|^2+|Z_s-\overline{Z}_s|^2\big)ds\\
	&\le C'\epsilon \mathbb{E}\Bigg[\int_t^Te^{\lambda s}\Big(1+2|X_s|^{l}+2\mathbb{E}|X_s|^l+|Y_s^k|+|\overline{Y}_s|^k\\
	&\hspace{3.5cm}+|Z_s|^k+|\overline{Z}_s|^k+\big(\mathbb{E}|Y_s|^2\big)^{1/2}+\big(\mathbb{E}|\overline{Y}_s|^2\big)^{1/2}\Big)^2ds\mathbbm{1}_{\mathcal{A}}\Bigg].
\end{aligned}\nonumber
\end{equation}
Sending $n\rightarrow\infty$, we have
\begin{equation} \label{eb39}
\begin{aligned}
	\mathbb{E}|Y_t-\overline{Y}_t|^2+\mathbb{E}\int_t^T\big(|Y_s-\overline{Y}_s|^2+|Z_s-\overline{Z}_s|^2\big)ds=0,
\end{aligned}\nonumber
\end{equation}
which together with the BDG inequality yields that
\begin{equation} \label{eb39}
\begin{aligned}
	\mathbb{E}\sup_{t\in[0,T]}|Y_t-\overline{Y}_t|^2+\mathbb{E}\int_0^T|Z_s-\overline{Z}_s|^2ds=0.\end{aligned}\nonumber
\end{equation}
From the uniqueness of $(Y,Z)$, we obtain the uniqueness of $\phi^{(2)}$. Therefore, the solution to equation \eqref{eb1} is unique.
\end{proof}

\section*{Declaration of competing interest}
The authors declare that they have no known competing financial interests or personal relationships that could have appeared to influence the work reported in this paper.

\section*{Acknowledgments}
The research of Jing Wu and Jinwei Zheng is supported by NSFC (No. 12071493). The authors would like to thank two anonymous reviewers and the Editors for their very constructive comments and efforts on this lengthy work, which greatly improved the quality of this paper.
%\end{document}

%\newpage			
%\appendix

\bibliography{bib-ms}

\end{document}